\documentclass{amsart}
\usepackage[utf8]{inputenc}
\usepackage{amsmath}
\usepackage{amsthm}
\usepackage{amssymb}
\usepackage{mathrsfs}
\usepackage{fullpage}
\usepackage{mathtools}
\usepackage{graphicx}
\usepackage{verbatim} 
\usepackage{tikz}
\usepackage{indentfirst}
\usepackage{enumitem}
\usepackage{algorithm}
\usepackage{algorithmic}

\renewcommand\vec[1]{\overrightarrow{#1}}
\newcommand\cev[1]{\overleftarrow{#1}}

\usepackage{xcolor}

\begin{document}
\title{Almost-nowhere intersection of Cantor sets, and sufficient sampling of their cumulative distribution functions}
\author{Allison Byars}
\address{Department of mathematics, North Dakota State University, 1210 Albrecht Boulevard, Fargo, ND 58102}
\email{allison.byars@ndsu.edu}
\author{Evan Camrud\textsuperscript{*}}
\address{Department of mathematics, Iowa State University, 411 Morrill Road, Ames, IA 50011}
\email{ecamrud@iastate.edu}
\thanks{\textsuperscript{*}Corresponding author}
\author{Steven N. Harding}
\address{Department of mathematics, Iowa State University, 411 Morrill Road, Ames, IA 50011}
\email{sharding@iastate.edu}
\author{Sarah McCarty}
\address{Department of mathematics, University of Nebraska at Omaha, 6001 Dodge Street
Omaha, NE 68182}
\email{smccarty@unomaha.edu}
\author{Keith Sullivan}
\address{Department of mathematics, Concordia College, 901 8th St. S. Moorhead, MN 56562}
\email{ksulliv1@cord.edu}
\author{Eric S. Weber}
\address{Department of mathematics, Iowa State University, 411 Morrill Road, Ames, IA 50011}
\email{esweber@iastate.edu}
\date{\today}
\begin{abstract}
Cantor sets are constructed from iteratively removing sections of intervals. This process yields a cumulative distribution function (CDF), constructed from the invariant measure associated with their iterated function systems. Under appropriate assumptions, we identify sampling schemes of such CDFs, meaning that the underlying Cantor set can be reconstructed from sufficiently many samples of its CDF.  To this end, we prove that two Cantor sets have almost-nowhere (with respect to their respective invariant measures) intersection. 
\end{abstract}
\maketitle
\newtheorem{theorem}{Theorem}[section]
\newtheorem{defn}{Definition}[section]
\newtheorem{conjecture}[theorem]{Conjecture}
\newtheorem*{counter}{Counterexample}
\newtheorem{exmp}{Example}
\newtheorem{lemma}{Lemma}[section]
\newtheorem{obs}{Observation}[section]
\newtheorem{cor}{Corollary}[section]
\newtheorem{remark}{Remark}[section]
\newtheorem{proposition}{Proposition}[section]

\section{Introduction and Motivations}

A Cantor set is the result of an infinite process of removing sections of an interval—$[0,1]$ in this paper—in an iterative fashion. The set itself consists of the points remaining after the removal of intervals specified by two parameters: the scale factor $N$ and digit set $D$. The positive integer $N$ determines how many equal intervals each extant segment is divided into per iteration, while $D \subset \{0,...,N-1\}$ enumerates which of the $N$ intervals of the segments will be preserved in each iteration. The Cantor set determined by such an $N$ and $D$ is denoted by $C_{N,D}$. Another notation to describe $N,D$ is to consider a vector $\vec{B}=(b_0,b_1,...,b_{N-1})\in (\mathbb{Z}_2)^N$, where $b_i= 1$ if $i \in D$ and $b_i = 0$ if $i \notin D$. We will also write $\vec{B}(i) \coloneqq b_i$. Further, $\|\vec{B}\|\coloneqq\sum_{i=0}^{N-1}b_i=|D|$. $\vec{B}$ is referred to as the binary representation, and we denote the Cantor set determined by $\vec{B}$ as $C_{\vec{B}}$. In this sense, both $C_{N,D}$ and $C_{\vec{B}}$ can be used to describe a Cantor set, and we naturally associate $N,D$ with its corresponding $\vec{B}$. Note that in this work, all indexing will start with zero. In addition, special cases exist in which a Cantor set will be considered degenerate. In particular, $C_{\vec{B}}$ is not considered when the set is empty, a one-point set, or $[0,1]$. Under this definition, there does not exist a Cantor set with $N<3$ or $\|\vec{B}\|$ equal to 0, 1, or $N$. For an example of a legitimate Cantor set, $C_{(1,0,1)}$ is the well-known ternary Cantor set (Figure 1). We also provide an illustration of the iterative construction of the Cantor set corresponding to $\vec{B}=(1,1,0,1)$ (Figure 2).

Each Cantor set yields a Cumulative Distribution Function (CDF), which we define formally in Definition \ref{D:CDF}.  We denote the class of all such CDFs by $\mathscr{F}$.  We consider the problems of sampling and interpolation of functions in $\mathscr{F}$.  By sampling, we mean the reconstruction of an unknown function $F \in \mathscr{F}$ from its samples $\{ F(x_{i}) \}_{i \in I}$ at known points $\{ x_{i} \}_{i \in I}$ in its domain (for an introduction to sampling theory, see \cite{BF01a,AG00a}).  By interpolation, we mean the construction of a function $F \in \mathscr{F}$ that satisfies the constraints $F(x_{i}) = y_{i}$ for \emph{a priori} given data $\{ (x_{i}, y_{i}) \}_{i \in I}$. Note that the premise of the sampling problem is that there is a unique $F \in \mathscr{F}$ that satisfies the available data, whereas the interpolation problem may not have the uniqueness property. Depending on the context, $I$ can be either finite or infinite. In this paper we focus on the finite case.

\begin{figure}[!ht]
  \centering
  \begin{minipage}[b]{0.4\textwidth}
    \includegraphics[width=\textwidth]{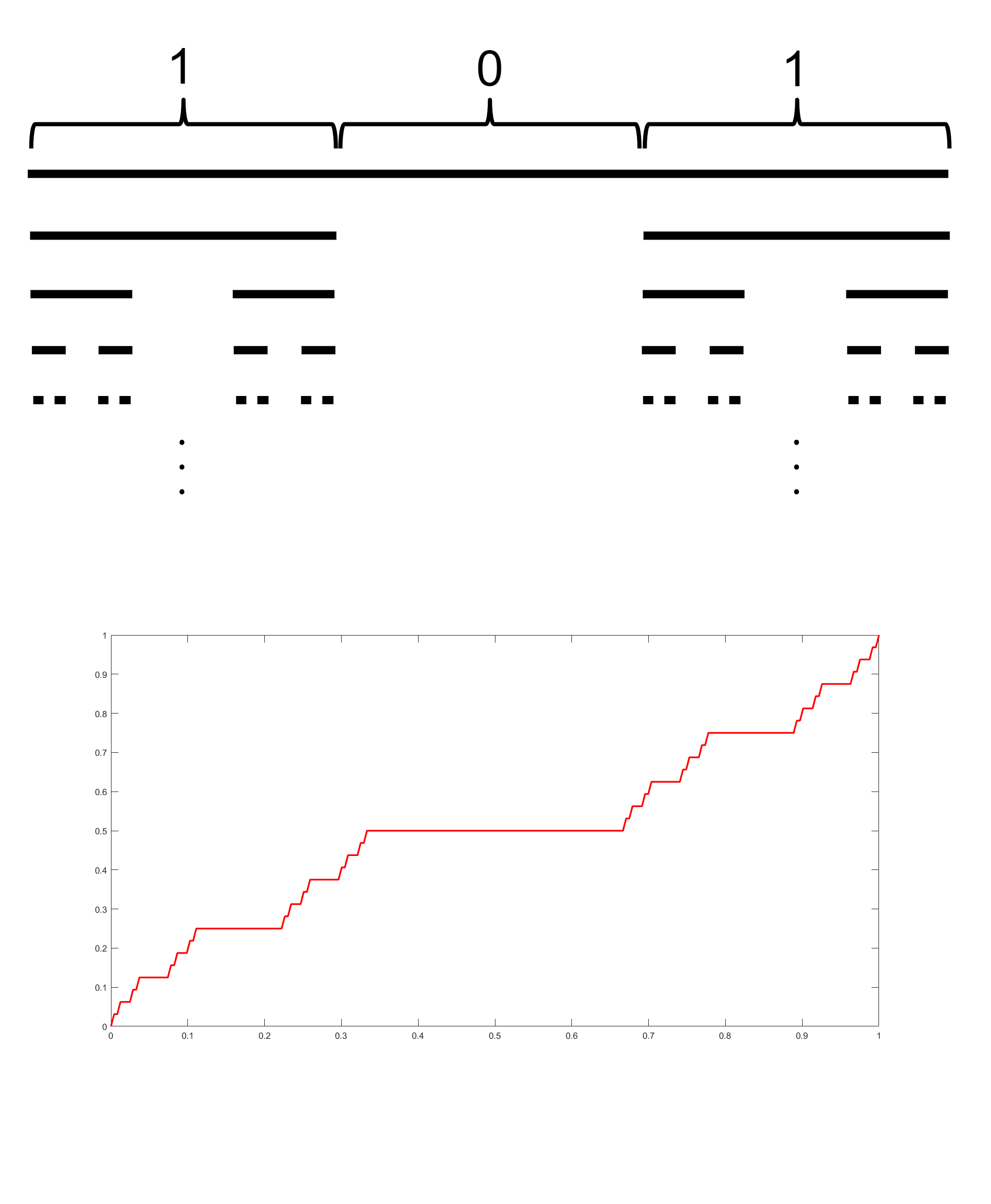}
    \caption{$C_{(1,0,1)}$ and $F_{(1,0,1)}$}
  \end{minipage}
  \hfill
  \begin{minipage}[b]{0.4\textwidth}
    \includegraphics[width=\textwidth]{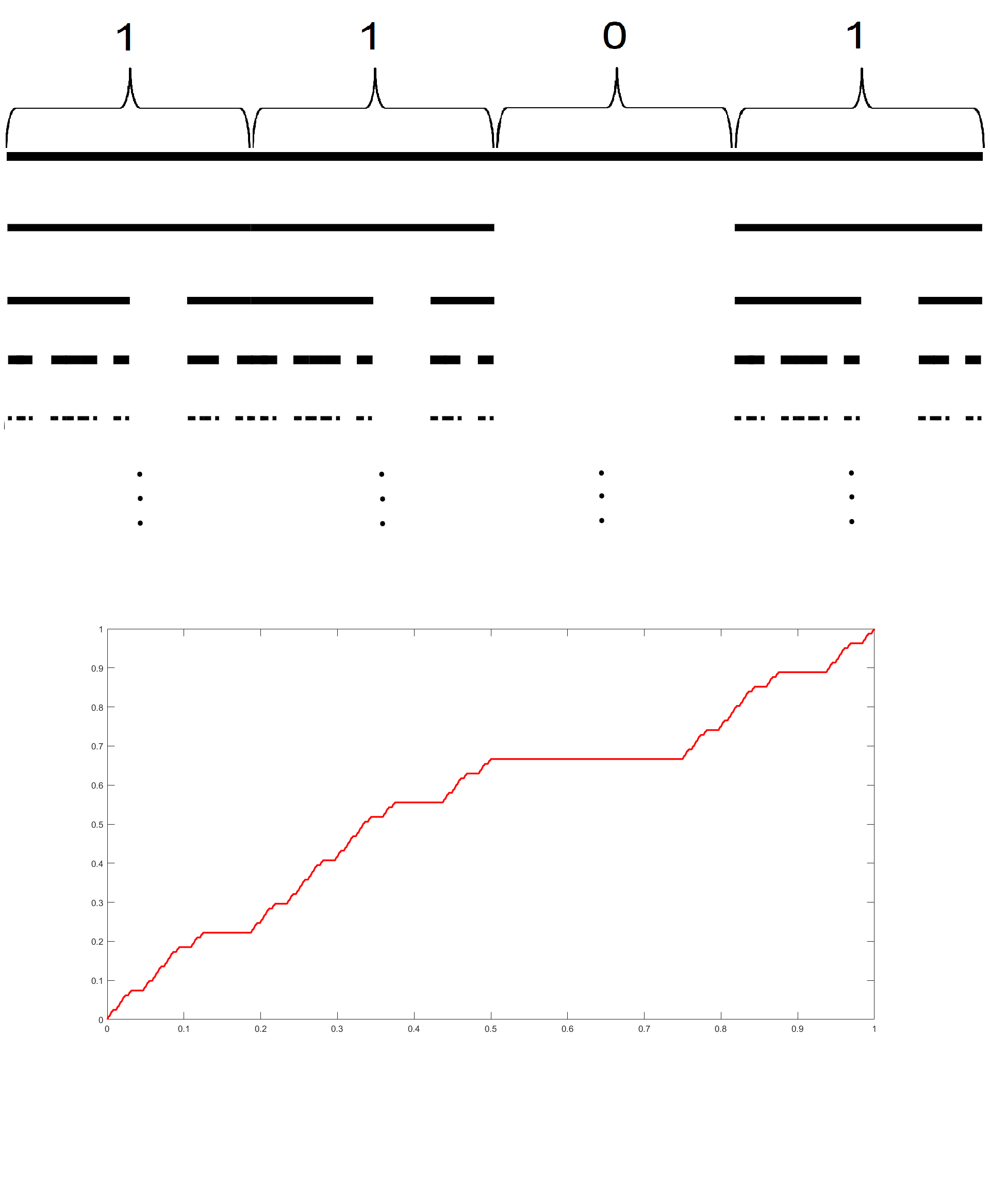}
    \caption{$C_{(1,1,0,1)}$ and $F_{(1,1,0,1)}$}
  \end{minipage}
\end{figure}

To be more precise regarding sampling CDFs, we formulate the problem as follows: Fix  $\mathscr{G} \subset \mathscr{F}$. For which sets of sampling points $\{ x_{i} \}_{i \in I}$ does the following implication hold: 
\begin{equation} \label{Eq:sampling}
F,G \in \mathscr{G} \text{ and } F(x_{i}) = G(x_{i}) \ \forall i \in I \Rightarrow F = G?
\end{equation}
In the case where (\ref{Eq:sampling}) holds, we call $\{x_i\}$ a \emph{set of uniqueness} for $\mathscr{G}$.

Sampling of functions with fractal spectrum was first investigated in \cite{HuaStr05a}.  In those papers, the authors consider the class of functions $F$ which are the Fourier transform of functions $f \in L^2(\mu)$.  Here, the measure $\mu$ is a fractal measure that is \emph{spectral}, meaning that the Hilbert space $L^2(\mu)$ possesses an orthonormal basis of exponential functions.  Similar sampling theorems are obtained in \cite{HW17a} without the assumption that the measure is spectral.  In higher dimensions, graph approximations of fractals (such as the Sierpinski gasket) are often considered; sampling of functions on such graphs has been considered in \cite{ObeStrStr03a,Wey16a}.

Our main results in the paper concerning sampling include the following.  In Theorem \ref{maintheorem} we prove that if $\mathscr{G}$ consists of all CDFs for Cantor sets with unknown scale factor $N$, but the scale factor is known to be bounded by $K$, then there exists a set of uniqueness of size $O(K^3)$.
We show that when the scale factor $N$ is known, there exists a set of uniqueness of size $N-1$ that satisfies the implication in Equation 1.  We conjecture that there is a minimal set of uniqueness of size $\left\lfloor \frac{N}{2}\right\rfloor$, and prove that the minimal set of uniqueness cannot be smaller in Proposition \ref{lessthanN/2}. We also provide evidence of our conjecture by considering a conditional sampling procedure (meaning that the sampling points are data dependent) that can uniquely identify the CDF from $\left\lfloor \frac{N}{2} \right\rfloor$ samples in Theorem \ref{conditional}.
Additionally, in section 2.2, we include an interpolation procedure as an imperfect reconstruction of a CDF from samples, and provide an upper bound on the error that the reconstruction via interpolation could give.

\subsection{Definitions}

Cantor sets, as defined by an iterative process, are naturally described in terms of an iterated function system (IFS). Indeed, the IFS encodes the iterative process that produces the Cantor set.

\begin{defn}[Iterated Function System]
In general, an IFS is a collection of contraction maps on a complete metric space. Then, the Cantor set $C_{\vec{B}}=C_{ N,D}$ for an IFS $\{\phi_{d}\}_{d \in D}$ is the attractor $C$ of the IFS, meaning the unique compact set satisfying
\[C = \bigcup_{d \in D}\phi_{d}(C).\]

We let $\phi_D(A)=\bigcup_{d\in D}\phi_d(A)$ and we will write \[(\phi_D)^n(A)\coloneqq\bigcup_{d_1,...,d_n\in D} \phi_{d_1}\circ...\circ\phi_{d_n}(A).\]

\end{defn}

Let $N$ be the scale factor, and let $D$ be the digit set. For our purpose, we consider the particular IFS $\{\phi_d\}_{d \in D}$ on $\mathbb{R}$ where $\phi_d(x) = \frac{x+d}{N}$ for each $d \in D$. We allow $\phi_D$ to act on $[0,1]$, so the invariant set is a subset of $[0,1]$. Note, $\bigcap_{n=1}^{\infty}(\phi_{D})^n([0,1])=C_{\vec{B}}$.

\begin{theorem}[Hutchinson, \cite{Hut81a}]\label{Hutch}
There exists a \emph{unique} probability measure $\mu_{\vec{B}}$ such that $\mu_{\vec{B}}(C_{\vec{B}})=1$, $\mu_{\vec{B}}([0,1]\backslash C_{\vec{B}})=0$ and $\mu_{\vec{B}}=\frac{1}{\|\vec{B}\|}\sum_{d\in D}\mu_{\vec{B}}\circ\phi_d^{-1}$. That is, $\mu_{\vec{B}}$ is invariant under the iterated function system.
\end{theorem}

\begin{defn}[Cumulative Distribution Function] \label{D:CDF}
Each Cantor set has a unique cumulative distribution function (CDF) $F: [0,1] \rightarrow [0,1]$ given by
\[F(x)=\mu_{\vec{B}}([0,x])=\int_0^xd\mu_{\vec{B}}.\]
The CDF of $C_{\vec{B}}$ is denoted $F_{\vec{B}}$.
\end{defn}

Note that the CDF of any of our Cantor sets is continuous.  When convenient, we will extend $F_{\vec{B}}$ to all of $\mathbb{R}$ by $F_{\vec{B}}(x) = 0$ if $x < 0$ and $1$ if $x > 1$.

\begin{remark}[Pullback of Lebesgue measure.]
For any subset $A$ of $C_{\vec{B}}$, $\mu_{\vec{B}}(A)=m(F_{\vec{B}}|_{C_{\vec{B}}}(A))$, where $m$ is Lebesgue measure.
\end{remark}

\begin{defn}[Kronecker Product of Binary Digit Vectors]
We recall the Kronecker product of two vectors: Let $\vec{B} = (b_0,b_1,...,b_{M-1})$ and $\vec{C} = (c_0,c_1,...,c_{N-1})$. Then, the \emph{Kronecker product} of $\vec{B}$ with $\vec{C}$, denoted $\vec{B} \otimes \vec{C}$, is defined as
\[ (\vec{B}\otimes \vec{C})(i)=b_{\left\lfloor \frac{i}{N} \right\rfloor}c_{i (mod N)},\] or equivalently, \[
(\vec{B} \otimes \vec{C})(n + mN) = b_m c_n
\]
where $n \in \{0,1,...,N-1\}$ and $m \in \{0,1,...,M-1\}$.  Further, we define the Kronecker product of two CDFs as follows: Let $F_{\vec{B}}$ and $F_{\vec{C}}$ be the CDFs corresponding to the binary digit vectors $\vec{B}$ and $\vec{C}$, respectively. The Kronecker product of $F_{\vec{B}}$ with $F_{\vec{C}}$, denoted $F_{\vec{B}} \otimes F_{\vec{C}}$, is the CDF whose binary digit vector is $\vec{B} \otimes \vec{C}$.
\end{defn}
We can define a Kronecker product on digit sets to retain the association of $\vec{B},\vec{C}$ with $N_1,D_1,N_2,D_2$.
\begin{defn}[Kronecker Product of Digit Sets]
 The \emph{Kronecker product} of two digit sets $D_1$ and $D_2$, denoted $D_1\otimes D_2$, is defined to be the Kronecker product of their associated binary digit vectors, reassociated with digit sets.
 \end{defn}
 
\begin{lemma}
$D_1\otimes D_2=\{c+bN_2 \mid c\in D_2,b\in D_1\}$ where $N_2$ is the scale factor corresponding to $D_2$. The scale factor associated to $D_1\otimes D_2$ is $N_1N_2$.
\end{lemma}
 
\begin{defn}
 $\vec{B}^{\otimes n}\coloneqq\overbrace{\vec{B}\otimes \vec{B}\otimes...\otimes \vec{B}}^{n\hspace{0.1cm} \text{times}}$. For example, $\vec{B}^{\otimes 1}=\vec{B}$, and $\vec{B}^{\otimes 2}=\vec{B}\otimes \vec{B}$.
 \end{defn}
\begin{defn}[Cumulative Digit Function]
Let $\vec{B}=(b_0,...,b_{N-1})$. Define $g:\{0,...,N\}\rightarrow \{0,...,\|\vec{B}\|\}$ to be the \emph{cumulative digit function} where $g(0)=0$ and $g(i)\coloneqq \sum\limits_{j=0}^{i-1}b_j\,\forall i\in\{1,...,N\}$.
\end{defn}

 Here, we describe an algorithm for approximating the CDF of a Cantor set. To be precise, we recursively define a sequence of piecewise linear functions $\{f_n\}$ which converges uniformly to the desired CDF. 
 
 \begin{defn}[Piecewise Approximations of CDFs]
  Let $C_{\vec{B}}$ be a Cantor set with cumulative digit function $g_{\vec{B}}$. Define $F_{\vec{B}}^{(1)}$ as the linear interpolation of the points
 \[S_1=
 \left\{\left(\frac{i}{N},\frac{g(i)}{\|\vec{B}\|}\right)\,\middle|\,i \in \{0,1,...,N\}\right\}.
 \]Let \[S_n=\left\{\left(\sum_{i=1}^n\frac{a_i}{N^i},\sum_{i=1}^n\frac{g(a_i)}{\|\vec{B}\|^i}\right)\,\middle|\,0\leq a_1\leq N,\, 0\leq a_i\leq N-1\, \forall i\in \{2,...,n\}\right\}\] and define $F_{\vec{B}}^{(n)}$ to be the linear interpolation of $S_n$.
 \end{defn}
It can be shown that
\[F_{\vec{B}}(x)=\lim_{n\to\infty}F_{\vec{B}}^{(n)}(x)\]
where the limit converges uniformly on $[0,1]$.

\begin{defn}[Multiplicative Dependence]
Two integers $r$ and $s$ are \emph{multiplicatively dependent}, denoted by $r\sim s$, if there exist integers $m$ and $n$ not both zero such that $r^m = s^n$. Else, if no such integers exist, then $r$ and $s$ are \emph{multiplicatively independent}, denoted by $r\not\sim s$. We note that $\sim$ is not an equivalence relation, e.g. $2^0 = 1^1 = 3^0$.
\end{defn}

We denote the exponential function $e^{2 \pi i x}$ by $e(x)$.

\section{Main Results}
\subsection{Preliminary Theorems}
The first Lemma of this section is a very useful invariance identity of the CDF. It is an immediate consequence of Theorem \ref{Hutch}, and we omit the proof.
\begin{lemma}[Invariance Equation]\label{invariance}
For a CDF $F_{\vec{B}}$ with scale factor $N$
\begin{equation}\label{E: invariance}
F_{\vec{B}}(x)=\sum_{n=0}^{N-1}\frac{b_n}{\|\vec{B}\|}F_{\vec{B}}(N x-n)\end{equation}
where we regard $F(x)=0$ for all $x\leq 0$ and $F(x)=1$ for all $x\geq 1$.
\end{lemma}

\begin{proof}
This follows nearly immediately from Theorem \ref{Hutch}, however, we present the proof anyway. Observe,
\[\begin{split}
    F_{\vec{B}}(x) &= \int_0^x d\mu_{\vec{B}}=\int_0^x d\sum_{d\in D}\frac{1}{\|D\|}\mu_{\vec{B}}\circ\phi_d^{-1}=\sum_{d\in D}\frac{1}{\|D\|}\int_0^x d\mu_{\vec{B}}\circ\phi_d^{-1}.
\end{split}\]
Hence under a change-of-variables
\[\begin{split}
   F_{\vec{B}}(x) &=\sum_{d\in D}\frac{1}{\|D\|}\int_{\phi_d^{-1}(0)}^{\phi_d^{-1}(x)}d\mu_{\vec{B}}\circ\phi_d^{-1}\circ\phi_d=\sum_{d \in D}\frac{1}{\|D\|}\int_{-d}^{Nx-d}d\mu_{\vec{B}}\\
    &=\sum_{d \in D}\frac{1}{\|D\|}\left(\int_{-d}^0d\mu_{\vec{B}}+\int_{0}^{Nx-d}d\mu_{\vec{B}}\right)=\sum_{d \in D}\frac{1}{\|D\|}\int_{0}^{Nx-d}d\mu_{\vec{B}}.
\end{split}\]
Finally, since $\|\vec{B}\|=\|D\|$, $D\subset\{0,1,...,N-1\}$, and $b_n=1$ for $n\in D$ and $b_n=0$ for $n\not\in D$,
\[\begin{split}
    F_{\vec{B}}(x) &= \sum_{d\in D}\frac{1}{\|\vec{B}\|}\int_{0}^{Nx-d}d\mu_{\vec{B}}=\sum_{n=0}^{N-1}\frac{b_n}{\|\vec{B}\|}\int_{0}^{Nx-d}d\mu_{\vec{B}}=\sum_{n=0}^{N-1}\frac{b_n}{\|\vec{B}\|}F_{\vec{B}}(Nx-d).
\end{split}\]
\end{proof}

\begin{lemma}
  $F_{\vec{B}}\left(\frac{k}{N}\right)=\frac{g(k)}{\|\vec{B}\|}$ for $k\in\{0,...,N\}$, where $g$ is the cumulative digit function.
\label{fixedpoints}
\end{lemma}
\begin{proof} Let $\vec{B}=(b_0,...,b_{N-1})$ be the binary digit vector for $F_{\vec{B}}$. Then by the invariance equation \ref{E: invariance},

    \[F_{\vec{B}}\left(\frac{k}{N}\right)=\sum\limits_{n=0}^{N-1}\frac{b_n}{\|\vec{B}\|}F_{\vec{B}}(k-n)=\sum\limits_{n=0}^{k-1}\frac{b_n}{\|\vec{B}\|}=\frac{g(k)}{\|\vec{B}\|}.\]
\end{proof}

\begin{proposition} 
A function $g:\{0,...,N\}\rightarrow \{0,...,d\}$ is a cumulative digit function for some valid CDF if and only if the following criteria are met.
\begin{enumerate}[nolistsep]
\item $g(0)=0$
\item $g(N)=d$, for some $d\in\{2,...,N-1\}$
\item $0\leq g(k+1)-g(k)\leq 1$ for all $k\in\{0,...,N-1\}$.
\end{enumerate}
Moreover, if $g$ satisfies conditions (1), (2), and (3), then the corresponding CDF $C_{\vec{B}}$ has binary representation $\vec{B}=(b_0,...,b_{N-1})$ such that $b_k=1$ if and only if $g(k+1)-g(k)=1$ and $\|\vec{B}\|=d$.
\end{proposition}
\begin{proof}
$(\Rightarrow)$ Let $g$ be the cumulative digit function for $C_{\vec{B}}$. The first condition follows directly from the definition of $g$. Also, $g(N)=\sum\limits_{j=0}^{N-1}b_j=d$ so the second condition holds.  By definition of $g$, $g(i)=\sum\limits_{j=0}^{i-1}b_j\leq \sum\limits_{j=0}^{i}b_j=g(i+1)$, so $0\leq g(i+1)-g(i)$\\
Finally, $g(i)+1=\sum\limits_{j=0}^{i-1}b_j+1\geq \sum\limits_{j=0}^{i}b_j=g(i+1)$ implies the third condition.

$(\Leftarrow)$ Construct a CDF with the binary representation $\vec{B}=(b_0,...,b_{N-1})$ such that  $b_k=1$ if and only if $g(k+1)-g(k)=1$. By the second and third conditions, at least two $b_i$ will be 1, and this is a valid CDF.\\
By the third condition and the range of $g$, either $g(k+1)-g(k)=1$ and $b_k=1$ or $g(k+1)-g(k)=0$ and $b_k=0$. By the first condition, $g(0)=0$. For induction, suppose that for $0\leq i\leq N-1$, $g(i)=\sum_{k=0}^{i-1}b_k$. Then, $g(i+1)-g(i)=1$ if and only if $b_i=1$. Therefore, $g(i+1)=g(i)+1=\sum_{k=0}^{i-1}b_k+1$ if and only if $b_i$ is 1. Then, $g(i+1)=\sum_{k=0}^{i}b_k$. By induction, it follows $g$ is the cumulative digit function of $\vec{B}$ by definition.

\end{proof}

\subsubsection{Kronecker Product Results}

We define $\phi_{D_1}\circ\phi_{D_2}(A)=\bigcup_{d\in D_1}\phi_{d}\big(\bigcup_{d'\in D_2}\phi_{d'}(A)\big)$.

\begin{proposition}
 Consider Cantor sets $C_{\vec{B}_1}$ and $C_{\vec{B}_2}$ such that the scale factor and digit set for $\vec{B}_i$ are $N_i,D_i$. Then
$\phi_{D_1}\circ\phi_{D_2}=\phi_{D_1\otimes D_2}$
\end{proposition}

\begin{proof}
First, $y\in\phi_{D_1}\circ\phi_{D_2}\left([0,1]\right)$ if and only if there exists $x\in[0,1]$ such that $y=\phi_{D_1}\circ\phi_{D_2}(x)$. This occurs if and only if \[y=\frac{\frac{x+\epsilon_2}{N_2}+\epsilon_1}{N_1}=\frac{x+\epsilon_2+\epsilon_1N_2}{N_1N_2}\] for some $\epsilon_1\in D_1$, $\epsilon_2\in D_2$.\\
This is the IFS for scale factor $N_1N_2$ and digit set $D_3=\{\epsilon_2+\epsilon_1N_2 \mid \epsilon_2\in D_2, \epsilon_1\in D_1\}=D_1\otimes D_2$ by definition of the Kronecker product.
\end{proof}
\begin{cor}$(\phi_D)^n=\phi_{D^{\otimes n}}$ for all $n\in \mathbb{Z}^+$.
\label{ISFcor1}
\end{cor}

\begin{cor}
$F_{\vec{B}}=F_{\vec{B}^{\otimes k}}$.
\label{ISFcor3}
\end{cor}

\begin{proof}
Since $F_{\vec{B}}$ is uniquely determined by $C_{\vec{B}}$, and $C_{\vec{B}}$ is uniquely determined by the property that $\phi_D(C_{\vec{B}})=C_{\vec{B}}$, we have that $C_{\vec{B}}=(\phi_D)^n(C_{\vec{B}})=\phi_{D^{\otimes n}}(C_{\vec{B}})$. Hence $C_{\vec{B}}$ satisfies the invariance property of $\phi_{D^{\otimes n}}$. Since $D^{\otimes n}$ was defined to retain its association with $\vec{B}^{\otimes n}$ we have that $F_{\vec{B}}=F_{\vec{B}^{\otimes n}}$.
\end{proof}


\begin{lemma}
Let $\vec{B}=(b_0,...,b_{M-1})$ be a binary representation with cumulative digit function $g_{\vec{B}}$, $\vec{C}=(c_0,..,c_{N-1})$ be a binary representation with cumulative digit function $g_{\vec{C}}$, and $g_{\vec{B}\otimes \vec{C}}$ be the cumulative digit function for $\vec{B}\otimes \vec{C}$. Then, for $j\in \{0,...,N\},$ $k\in \{0,...,M\},$ $g_{\vec{B}\otimes \vec{C}}(kN+j)=\|\vec{C}\|g_{\vec{B}}(k)+b_kg_{\vec{C}}(j)$.
\label{kroneckerg}
\end{lemma}
\begin{proof}
The proof follows by induction on $j$.\\
When $j=k=0$, $g_{\vec{B}\otimes\vec{C}}(0)=0=\|\vec{C}\|g_{\vec{B}}(0)$ by definition. When $k\geq 1$, then
\begin{align*}
g_{\vec{B}\otimes \vec{C}}(kN) &= \sum_{i=0}^{kN-1}(\vec{B}\otimes \vec{C})(i)= \sum_{m=0}^{k-1}\sum_{n=0}^{N-1}(\vec{B}\otimes \vec{C})(n + mN) \\
&= \sum_{m=0}^{k-1}\sum_{n=0}^{N-1} b_mc_n= \sum_{n=0}^{N-1}c_n\sum_{m=0}^{k-1}b_m \\
&= \|\vec{C}\|g_{\vec{B}}(k)
\end{align*}
as desired.\\
It follows the identity holds for all $k$ when $j=0$. This serves as the base case for induction on $j$.
Now assume the identity for $j$.
Then,
\begin{align*}
g_{\vec{B} \otimes \vec{C}}(kN + j + 1) &= g_{\vec{B} \otimes \vec{C}}(kN + j) + (\vec{B} \otimes \vec{C})(kN + j) \\
&= \|\vec{C}\|g_{\vec{B}}(k) + b_kg_C(j) + (\vec{B} \otimes \vec{C})(kN + j) \\
&=\|\vec{C}\|g_{\vec{B}}(k)+b_kg_{\vec{C}}(j)+b_kc_{j}.
\end{align*}
It follows, when $b_k=1$,
\[g_{\vec{B} \otimes \vec{C}}(kN + j + 1)=\|\vec{C}\|g_{\vec{B}}(k)+g_{\vec{C}}(j)+c_{j}= \|\vec{C}\|g_{\vec{B}}(k) + b_kg_{\vec{C}}(j + 1).\]
Otherwise, when $b_k=0$,
\[g_{\vec{B} \otimes \vec{C}}(kN + j + 1)= \|\vec{C}\|g_{\vec{B}}(k) =\|\vec{C}\|g_{\vec{B}}(k) + b_kg_{\vec{C}}(j+1).\]
\end{proof}

\begin{proposition}
Let $F_{\vec{B}}$ be a CDF where $\vec{B}=(b_0,b_1,...,b_{N-1})$ with cumulative digit function $g_{\vec{B}}$. \\
Consider $x\in(0,1)$, with $x=\sum\limits_{i=1}^{\infty}\frac{n_i}{N^i}$, $n_i\in\mathbb{Z}_N$.
Then, $F_{\vec{B}}(x)=\sum\limits_{i=1}^{\infty}\left(\prod\limits_{k=1}^{i-1}b_{n_k}\right)\frac{g(n_i)}{\|\vec{B}\|^i}$.
\label{NtoDrel}
\end{proposition}
\begin{proof}
Fix the sequence $\{ n_i \} \subset \mathbb{Z}_N$.
We have, by Corollary \ref{ISFcor3} and Lemma \ref{fixedpoints}, for all $j \in \mathbb{N}$
\[F_{\vec{B}}\left(\sum_{i=1}^j\frac{n_i}{N^i}\right)=F_{ \vec{B}^{\otimes j}}\left(\sum_{i=1}^j\frac{n_i}{N^i}\right)=\frac{g_{ \vec{B}^{\otimes j}}\left(\sum_{i=1}^j N^{j-i}n_i\right)}{\|\vec{B}\|^j}.\]

For an inductive base case, by Proposition \ref{fixedpoints}, \[F_{\vec{B}}\left(\frac{n_1}{N}\right)=\frac{g_{\vec{B}}(n_1)}{\|\vec{B}\|}=\sum_{i=1}^1\left(\prod_{k=1}^{i-1}b_{n_k}\right)\frac{g_{\vec{B}}(n_i)}{\|\vec{B}\|^i}.\]
For induction on $j$, suppose that  \[F_{\vec{B}}\left(\sum_{i=1}^j\frac{n_i}{N^i}\right)=\sum_{i=1}^j\left(\prod_{k=1}^{i-1}b_{n_k}\right)\frac{g_{\vec{B}}(n_i)}{\|\vec{B}\|^i}.\]
Then, with Lemma \ref{kroneckerg} and Lemma \ref{ISFcor3},
\begin{align*}
F_{\vec{B}}\left(\sum_{i=1}^{j+1}\frac{n_i}{N^i}\right)&=F_{\vec{B}^{\otimes j+1}}\left(\sum_{i=1}^{j+1}\frac{n_i}{N^i}\right)=F_{\vec{B}^{\otimes j+1}}\left(\frac{\sum_{i=1}^{j+1}n_iN^{j+1-i}}{N^{j+1}}\right)\\
&=\frac{g_{\vec{B}^{\otimes j+1}}\left(\sum_{i=1}^{j+1}n_i N^{j+1-i}\right)}{\|\vec{B}\|^{j+1}}=\frac{g_{\vec{B}^{\otimes j}\otimes \vec{B}}\left(n_{1}N^{j}+\sum_{i=2}^{j+1}n_i N^{j+1-i}\right)}{\|\vec{B}\|^{j+1}}\\
&=\frac{\|\vec{B}\|^jg_{\vec{B}}(n_1)+b_{n_1}g_{ \vec{B}^{\otimes j}}\left(\sum_{i=2}^{j+1}n_iN^{j+1-i}\right)}{\|\vec{B}\|^{j+1}}=\frac{g_{\vec{B}}(n_1)}{\|\vec{B}\|}+\frac{b_{n_1}}{\|\vec{B}\|}\frac{g_{\vec{B}^{\otimes j}}\left(\sum_{i=2}^{j+1}n_iN^{j+1-i}\right)}{\|\vec{B}\|^j}\\
&=\frac{g_{\vec{B}}(n_1)}{\|\vec{B}\|}+\frac{b_{n_1}}{\|\vec{B}\|}\sum_{i=2}^{j+1}\left(\prod_{k=2}^{i-1}b_{n_k}\right)\frac{g_{\vec{B}}(n_i)}{\|\vec{B}\|^{i-1}}\text{, by shifting indices in the inductive hypothesis}\\
&=\frac{g_{\vec{B}}(n_1)}{\|\vec{B}\|}+\sum_{i=2}^{j+1}\left(\prod_{k=1}^{i-1}b_{n_k}\right)\frac{g_{\vec{B}}(n_i)}{\|\vec{B}\|^{i}}=\sum_{i=1}^{j+1}\left(\prod_{k=1}^{i-1}b_{n_k}\right)\frac{g_{\vec{B}}(n_i)}{\|\vec{B}\|^{i}}.
\end{align*}
Thus, by induction, for all $j$ and $n_i\in\mathbb{Z}_N$
\[
F_{\vec{B}} \left(\sum_{i=1}^j\frac{n_i}{N^i}\right)=\sum_{i=1}^j\left(\prod_{k=1}^{i-1}b_{n_k}\right)\frac{g_{\vec{B}}(n_i)}{\|\vec{B}\|^{i}}.
\]
Next, note all $x\in(0,1)$ have the form $\sum_{i=1}^\infty\frac{n_i}{N^i}$ for some $n_i\in\mathbb{Z}_N$. Since $F_{\vec{B}}$ is a continuous function, \[F_{\vec{B}}(x)=\lim_{j\rightarrow\infty}F_{\vec{B}}\left(\sum_{i=1}^j\frac{n_i}{N^i}\right)=\lim_{j\rightarrow\infty}\sum_{i=1}^{j}\left(\prod_{k=1}^{i-1}b_{n_k}\right)\frac{g_{\vec{B}}(n_i)}{\|\vec{B}\|^{i}}=\sum_{i=1}^{\infty}\left(\prod_{k=1}^{i-1}b_{n_k}\right)\frac{g_{\vec{B}}(n_i)}{\|\vec{B}\|^{i}}.\]
\end{proof}

\subsection{Interpolation}
\begin{proposition}\label{interpol}
Let $\{(x_n,y_n)\}_{n=1}^k \subset (\mathbb{Q}\cap(0,1))\times (\mathbb{Q}\cap(0,1))$, i.e. rational pairs in the unit cube, with $x_m \neq x_n$ for $m \neq n$ and $y_m \geq y_n$ whenever $x_m \geq x_n$.  Then there exists a CDF that interpolates the data $\{(x_n,y_n)\}_{n=1}^k$; i.e. there exists a digit set $\vec{B}$ such that $F_{\vec{B}}(x_n) = y_n$ for all $n$.
\end{proposition}
\begin{proof}
We may assume without loss of generality that $x_1 < x_2 < ... < x_k$. Further, by considering equivalent fractions, we may assume for all $n$, that $x_n = \frac{a_n}{N}$ and $y_n = \frac{c_n}{C}$ where $a_{i+1} - a_i \geq c_{i+1} - c_i + 1$ for $0 \leq i \leq k$ with the following conventions: $a_0 = c_0 = 0$, $a_{k+1} = N$, and $c_{k+1} = C$. We construct the digit set $\vec{B}$ of length $N$ as follows:
\begin{align*}
\vec{B}(a_i) &= \vec{B}(a_i+1) = ... = \vec{B}(a_i+c_{i+1}-c_i-1) = 1 \\
\vec{B}(a_i+c_{i+1}-c_i) &= \vec{B}(a_i+c_{i+1}-c_i+1) = ... = \vec{B}(a_{i+1}-1) = 0.
\end{align*}
Then, we observe the recurrence relation,
\begin{align*}
F_{\vec{B}}(x_1) &= F_{\vec{B}}\left(\frac{a_1}{N}\right) = \frac{c_1}{C} = y_1 \\
F_{\vec{B}}(x_{i+1}) - F_{\vec{B}}(x_i) &= F_{\vec{B}}\left(\frac{a_{i+1}}{N}\right) - F_{\vec{B}}\left(\frac{a_i}{N}\right) = \frac{g_{\vec{B}}(a_{i+1}) - g_{\vec{B}}(a_i)}{C} = \frac{c_{i+1} - c_i}{C} = y_{i + 1} - y_i
\end{align*}
which concludes the proof.

\end{proof}

\begin{remark}
Let $\{(x_n,y_n)\}$ be a finite sampling set of rational pairs in the unit cube satisfying the hypotheses of Proposition \ref{interpol}. We note from the proof of the proposition that interpolation by a CDF is not unique.
\end{remark}

\begin{cor}
Let $\{(x_n,y_n)\}_{n=1}^k \subset (0,1)\times(\mathbb{Q}\cap(0,1))$ with $x_m \neq x_n$ for $m \neq n$ and $y_m \geq y_n$ whenever $x_m \geq x_n$.  Then there exists a CDF that interpolates the data $\{(x_n,y_n)\}_{n=1}^k$; i.e. there exists a digit set $\vec{B}$ such that $F_{\vec{B}}(x_n) = y_n$ for all $n$.
\end{cor}

\begin{proof}
We may assume without loss of generality that $0 < x_1 < x_2 < ... < x_k < 1$. Now select a collection of rational pairs $\{(z_n,w_n)\}_{n=1}^{2k}$ such that $z_1 < x_1$, $x_n < z_{2n} < z_{2n+1} < x_{n+1}$ for $1 \leq n \leq k - 1$, $x_k < z_{2k}$, and $w_{2n-1} = w_{2n} = y_n$ for all $n$. Then, by Proposition \ref{interpol}, there exists a digit set $\vec{B}$ such that $F_{\vec{B}}(z_n) = w_n$ for all $n$ and, in particular, $F_{\vec{B}}(x_n) = y_n$.

\end{proof}

\begin{cor}
Let $\{(x_n,y_n)\}_{n=1}^k$ be a set of samples of a CDF. Then the maximum error in the reconstruction is
\[\max_{n=1,...,k-1}(y_{n+1}-y_{n}).\]
\end{cor}

\begin{proof}
Without loss of generality, let $(x_1,y_1),(x_2,y_2)$ be such that
\[(y_2-y_1)=\max_{n=1,...,k-1}(y_{n+1}-y_n).\]
Then there exist a sequence of CDFs $F_{\vec{B}_n}$ such that
\[\lim_{n\to\infty}F_{\vec{B}_n}\Big(x_2-\frac{1}{n}\Big)=y_1.\]
\end{proof}

\subsection{Sampling}

We first show that if we know the scaling factor $N$, then $N-1$ well chosen sample points is enough to reconstruct $F_{\vec{B}}$. 
 
 \begin{lemma}
For $m\in\{0,...,N-1\}$, $F_{\vec{B}}\left(\frac{m+1}{N}\right)=F_{\vec{B}}\left(\frac{m}{N}\right)$ if and only if $b_m=0$.
\label{basicCDF}
\end{lemma}

\begin{proof} Let $\vec{B}=(b_0,...,b_{N-1})$ be the binary digit vector for $F_{\vec{B}}$.
By Lemma \ref{fixedpoints}, $F\left(\frac{m+1}{N}\right)-F\left(\frac{m}{N}\right)=\frac{b_m}{\|\vec{B}\|}$. Then, $F\left(\frac{m+1}{N}\right)=F\left(\frac{m}{N}\right)$ if and only if $b_m=0$.
\end{proof}

\begin{theorem}
Let $F_{\vec{B}}$ be a CDF with $\|\vec{B}\|=N$. Given $\{F_{\vec{B}}(\frac{k}{N})\}_{k=1}^{N-1}$, $\vec{B}$ can be uniquely determined.
\label{uniqueDet}
\end{theorem}
\begin{proof}
Since $F_{\vec{B}}(0)=0$ and $F_{\vec{B}}(1)=1$, this follows from Lemma \ref{basicCDF}.
\end{proof}

\begin{cor}
If $\mathscr{G}_N=\{F_{\vec{B}}:\|\vec{B}\|=N\}$, then $\{\big(\frac{k}{N}\big):k=1,...,N-1\}$ is a set of uniqueness for $\mathscr{G}_N$.
\end{cor}

We will now consider the case when we do not know the scale factor.
\subsubsection{Motivating a bound on scale factor}
Remark 2.1 and Corollary 2.4 together establish that finite samples will never suffice without some sort of constraint. We contrast this with with Proposition 2.6 below as this shows a lower bound of $O(N)$ points is necessary, where $N$ is the scale factor.
The following proposition shows that to be able to uniquely determine a CDF with a finite number of points, there must be a bound on the scale factor.

\begin{lemma}
Fix an integer $N\geq 4$, and suppose $\{x_n\}_{1\leq n\leq k}\subset[0,1]$ where $0\leq x_{n-1}\leq x_n \leq 1$ for all $n$ and $k<\left\lfloor\frac{N}{2}\right\rfloor$. There exist two distinct CDFs $F_{\vec{B}}$ and $F_{\vec{C}}$, both with scale factor $N$, such that $F_{\vec{B}}(x_n)=F_{\vec{C}}(x_n)$ $\forall n\in \{1,...,k\}$.
\label{lessthanN/2}
\end{lemma}

\begin{proof}
First, we note that there exists an integer $i$ such that $x_n \notin \left(\frac{i}{N},\frac{i+2}{N}\right)$ for all $n \in \{1,2,...,k\}$ by the pigeon-hole principle.

Next, since $N \geq 4$, we also have that there exists an integer $j \in \{0,1,...,N-1\}\setminus\{i,i+1\}$ such that $x_n \notin \left(\frac{j}{N},\frac{j+1}{N}\right)$ for all $n \in \{1,2,...,k\}$.
\\\\
We construct two distinct digit sets $\vec{B}=(b_0,b_1,...,b_{N-1})$ and $\vec{C}=(c_0,c_1,...,c_{N-1})$ as follows: Let $b_i=0$, $b_{i+1}=1$, $c_i=1$, $c_{i+1}=0$, $b_j=c_j=1$, and $b_m=c_m=0$ for all $m \notin \{i,i+1,j\}$. Note that both digit sets are nondegenerate since two digits are kept and $\|\vec{B}\|=\|\vec{C}\|=2$. We note that since $\|\vec{B}\|=\|\vec{C}\|$ and $\vec{B}\neq\vec{C}$, then $F_{\vec{B}}\neq F_{\vec{C}}$. We conclude the proof by showing that $F_{\vec{B}}(x_n) = F_{\vec{C}}(x_n)$ for all $n$.
\\\\
\textbf{Case 1:} $i<j$ \\
Let $x \leq \frac{i}{N}$. Then by Lemma \ref{fixedpoints}
\[
0 \leq F_{\vec{B}}(x) \leq F_{\vec{B}}\left(\frac{i}{N}\right) = \frac{g_{\vec{B}}(i)}{2} = 0.
\]
Likewise, $F_{\vec{C}}(x) = 0$. Now let $\frac{i + 2}{N} \leq x \leq \frac{j}{N}$. Then
\[
\frac{1}{2} = \frac{g_{\vec{B}}(i+2)}{2} = F_{\vec{B}}\left(\frac{i + 2}{N}\right) \leq F_{\vec{B}}(x) \leq F_{\vec{B}}\left(\frac{j}{N}\right) = \frac{g_{\vec{B}}(j)}{2} = \frac{1}{2}.
\]
Likewise, $F_{\vec{C}}(x) = \frac{1}{2}$. Finally let $\frac{j+1}{N} \leq x \leq 1$. Then
\[
1 = \frac{g_{\vec{B}}(j+1)}{2} = F_{\vec{B}}\left(\frac{j+1}{N}\right) \leq F_{\vec{B}}(x) \leq 1.
\]
Likewise, $F_{\vec{C}}(x) = 1$. Thus, $F_{\vec{B}}(x_n) = F_{\vec{C}}(x_n)$ for all $n$.
\\\\
\textbf{Case 2:} $j<i$ \\
The argument is analogous to the one given for case 1, and we omit the details. \\\\
Figures \ref{N/2prooffig1} and \ref{N/2prooffig2} depict cases 1 and 2, respectively.
\begin{figure}[htp]\label{distinctCDF1}
\centering
\includegraphics[width=8cm]{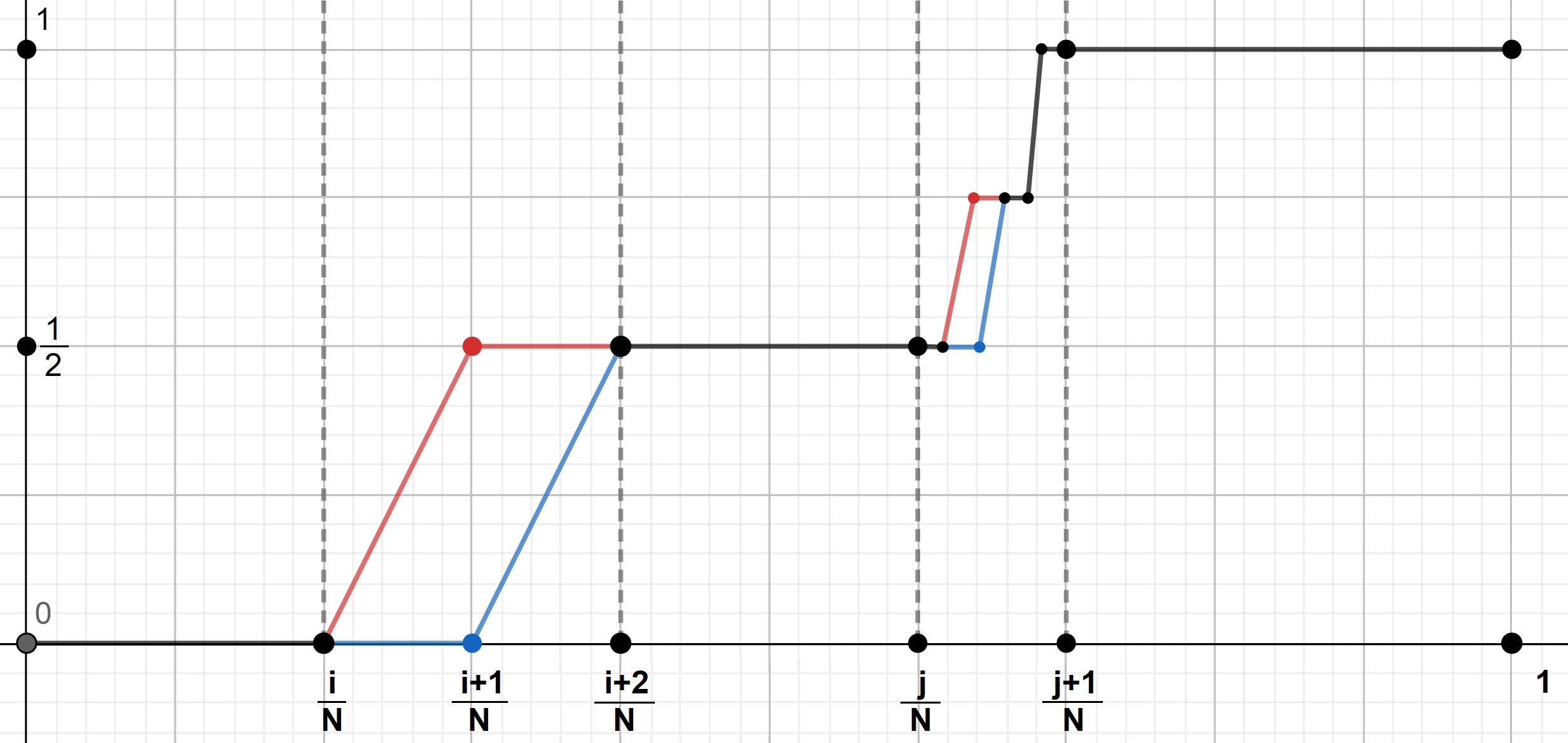}
\caption{Case 1 --- Sketch of piecewise linear approximations of $F_{\protect\vec{B}}$ (blue) and $F_{\protect\vec{C}}$ (red)}
\label{N/2prooffig1}
\end{figure}
\begin{figure}[htp]\label{distinctCDF2}
    \centering
    \includegraphics[width=8cm]{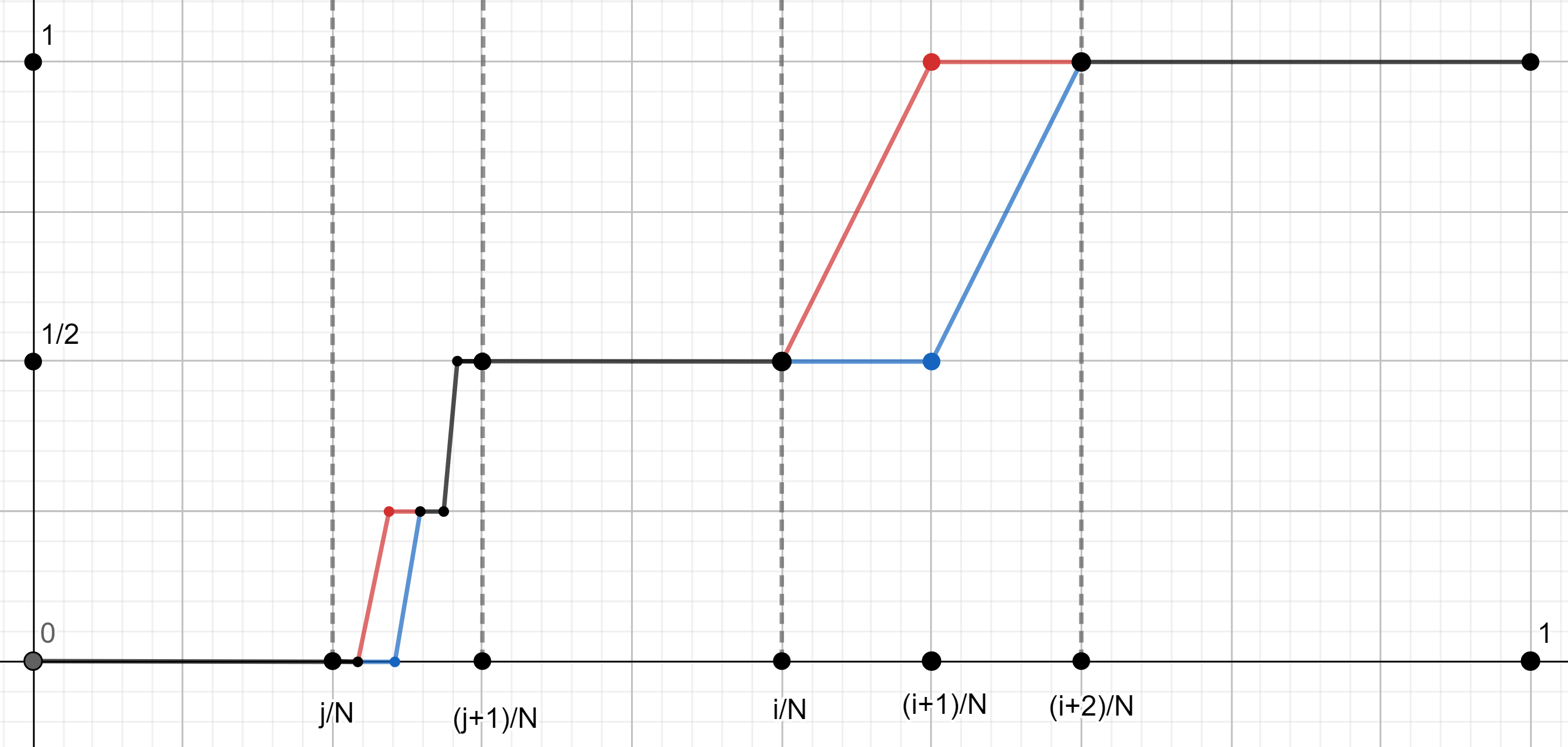}
    \caption{Case 2 --- Sketch of piecewise linear approximations of $F_{\protect\vec{B}}$ (blue) and $F_{\protect\vec{C}}$ (red)}
   \label{N/2prooffig2}
\end{figure}

\end{proof}

The next proposition observes the relationship between the CDFs of the digit set $\vec{B}$ and its reverse $\cev{B}$, that is $\cev{B}(n) = \vec{B}(N - 1 - n)$ for all $n$ where $N$ is the length of $\vec{B}$.

\begin{proposition}\label{reverse}
Let $\vec{B}$ be a digit set. Then,
\[
F_{\cev{B}}(x) = 1-F_{\vec{B}}(1 - x).
\]
\end{proposition}
\begin{proof}
Since $F_{\cev{B}}(x) + F_{\vec{B}}(1 - x)$ is continuous, it suffices to show the equality on a dense subset of the unit interval. Specifically, we show the identity on the set of $N$-adic numbers, that is
\[
\left\{\frac{1}{N^k}\sum_{\ell=0}^{k-1}n_\ell N^\ell\,\middle|\,k \in \mathbb{N}, n_\ell \in \{0,1,...,N-1\}\right\},
\]
where $N$ is the length of $\vec{B}$. We first observe that the simplest case, when $k = 1$, holds.
\[
F_{\cev{B}}\left(\frac{n_0}{N}\right) + F_{\vec{B}}\left(1 - \frac{n_0}{N}\right) = F_{\cev{B}}\left(\frac{n_0}{N}\right) + F_{\vec{B}}\left(\frac{N - n_0}{N}\right) = \sum_{n=0}^{n_0-1}\frac{b_{N - 1 - n}}{\|\vec{B}\|} + \sum_{n=0}^{N-n_0-1}\frac{b_n}{\|\vec{B}\|} = 1.
\]
We proceed by induction on the power of the $N$-adic number, assuming the identity is true for $k$. Then, by Lemma \ref{invariance},
\begin{align*}
&F_{\cev{B}}\left(\frac{1}{N^{k+1}}\sum_{\ell=0}^kn_\ell N^\ell\right) + F_{\vec{B}}\left(1 - \frac{1}{N^{k+1}}\sum_{\ell=0}^kn_\ell N^\ell\right)\\
&= F_{\cev{B}}\left(\frac{1}{N^{k+1}}\sum_{\ell=0}^kn_\ell N^\ell\right) + F_{\vec{B}}\left(\frac{1}{N^{k+1}} + \frac{1}{N^{k+1}}\sum_{\ell=0}^k(N - 1 - n_\ell) N^\ell\right) \\
&= \sum_{n=0}^{N-1}\frac{b_{N - 1 - n}}{\|\vec{B}\|}F_{\cev{B}}\left(n_k - n + \frac{1}{N^k}\sum_{\ell=0}^{k-1}n_\ell N^\ell\right) + \frac{b_n}{\|\vec{B}\|}F_{\vec{B}}\left(N - 1 - n_k - n + \frac{1}{N^k} + \frac{1}{N^k}\sum_{\ell=0}^{k-1}(N - 1 - n_\ell) N^\ell\right) \\
&= \frac{\|\vec{B}\| - b_{N - 1 - n_k}}{\|\vec{B}\|} + \frac{b_{N - 1 - n_k}}{\|\vec{B}\|}\left[F_{\cev{B}}\left(\frac{1}{N^k}\sum_{\ell=0}^{k-1}n_\ell N^\ell\right) + F_{\vec{B}}\left(\frac{1}{N^k} + \frac{1}{N^k}\sum_{\ell=0}^{k-1}(N - 1 - n_\ell) N^\ell\right)\right] \\
&= \frac{\|\vec{B}\| - b_{N - 1 - n_k}}{\|\vec{B}\|} + \frac{b_{N - 1 - n_k}}{\|\vec{B}\|}\left[F_{\cev{B}}\left(\frac{1}{N^k}\sum_{\ell=0}^{k-1}n_\ell N^\ell\right) + F_{\vec{B}}\left(1 - \frac{1}{N^k}\sum_{\ell=0}^{k-1}n_\ell N^\ell\right)\right] \\
&= \frac{\|\vec{B}\| - b_{N - 1 - n_k}}{\|\vec{B}\|} + \frac{b_{N - 1 - n_k}}{\|\vec{B}\|} = 1.
\end{align*}
Thus, the identity holds on the $N$-adic numbers, and the proof is done.

\end{proof}

We say that a sampling algorithm is \emph{conditional} if previously attained samples inform the selection of the next sample. For the remainder of this section, we describe a conditional sampling algorithm that completely determines a digit set $\vec{B}$ given its scale factor $N$. The algorithm as stated below requires at most $\left\lfloor\frac{N}{2}\right\rfloor$ samples to execute successfully which we note is the minimum number of samples that is required under non-conditional sampling to discern digit sets of equal scale factor. We first state the result.

\begin{theorem}\label{conditional}
Fix an integer $N \geq 3$, and let $\vec{B} = (b_0,b_1,...,b_{N-1})$ be a digit set with $2 \leq \|\vec{B}\| \leq N - 1$. Then there is a conditional sampling algorithm with at most $\left\lfloor\frac{N}{2}\right\rfloor$ points that completely determines $F_{\vec{B}}$.
\end{theorem}

The conditional sampling algorithm that answers Theorem \ref{conditional} is located in the appendix and split into two parts. Each part considers pairs of digits from $\vec{B}$ at a time, e.g. $(b_0,b_1)$, $(b_2,b_3)$, etc. The role of Algorithm 1 is to find the first nonzero digit of $\vec{B}$. As a consequence of the method, we can also find $\|\vec{B}\|$ from the sampling in Algorithm 1. Then the algorithm terminates if the first nonzero digit occurred in the last pair, i.e. $(b_{N-2},b_{N-1})$ if $N$ is even or $(b_{N-3},b_{N-2})$ if $N$ is odd, as $\vec{B}$ is then completely determined; otherwise, Algorithm 2 applies a similar procedure to $\cev{B}$. The sampling in Algorithm 2 is expressed in terms of $F_{\cev{B}}$ which translates to a sampling of $F_{\vec{B}}$ by Proposition \ref{reverse}. Then the maximum number of samples from both Algorithm 1 and Algorithm 2 is precisely the number of paired digits, that is there are at most $\left\lfloor\frac{N}{2}\right\rfloor$ samples. In the proof of the Theorem \ref{conditional}, we show that there exists a positive integer $\ell$ that is only dependent on $N$ (the smallest positive $\ell$ such that $2^{\ell + 1} > N - 1$ is sufficient) such that Algorithm 1 and Algorithm 2 are well-defined and completely determine $\vec{B}$.

\begin{proof}
Let $m \in \left\{1,2,...,\left\lfloor\frac{N}{2}\right\rfloor\right\}$. For convenience, we denote
\[
\psi_m(x) = \frac{g_{\vec{B}}(2m - 1)}{\|\vec{B}\|} + \frac{b_{2m - 1}}{\|\vec{B}\|}x,
\]
and use the notation $\psi^{\ell}_m = \psi_m\circ ...\circ\psi_m$ to represent the composition of $\ell$ functions. We claim that
\begin{align}
F_{\vec{B}}\left(\frac{2m}{N^{\ell+1}} + \sum_{n=1}^\ell\frac{2m-1}{N^n}\right) = \psi_m^{\ell + 1}(1). \label{compositions}
\end{align}
The case when $\ell = 0$ immediately follows from Lemma \ref{fixedpoints} since
\[
F_{\vec{B}}\left(\dfrac{2m}{N}\right) = \dfrac{g_{\vec{B}}(2m)}{\|\vec{B}\|} = \psi_m(1).
\]
To prove identity (\ref{compositions}) in general, we proceed by induction, so assume that the identity holds for $\ell$. Then, by Lemma \ref{invariance}, we find
\begin{align*}
F_{\vec{B}}\left(\frac{2m}{N^{\ell+2}} + \sum_{n=1}^{\ell+1}\frac{2m-1}{N^n}\right) &= \sum_{k=0}^{N-1}\dfrac{b_k}{\|\vec{B}\|}F_{\vec{B}}\left(\frac{2m}{N^{\ell+1}} + \left[\sum_{n=1}^\ell\frac{2m-1}{N^n}\right] + 2m - 1 - k\right) \\
&= \left[\sum_{k=0}^{2m - 2}\dfrac{b_k}{\|\vec{B}\|}\right] + \dfrac{b_{2m-1}}{\|\vec{B}\|}F_{\vec{B}}\left(\frac{2m}{N^{\ell+1}} + \sum_{n=1}^\ell\frac{2m-1}{N^n}\right) \\
&= \dfrac{g_{\vec{B}}(2m-1)}{\|\vec{B}\|} + \dfrac{b_{2m-1}}{\|\vec{B}\|}\psi_m^{\ell + 1}(1) \\
&= \psi_m^{\ell + 2}(1),
\end{align*}
as desired.

There are four cases to consider:
\\\\
\textbf{Case 1:} $b_{2m - 2} = b_{2m - 1} = 0$. Then
\[
\psi_m^{\ell + 1}(1) = \frac{g_{\vec{B}}(2m - 2)}{\|\vec{B}\|}.
\]
\textbf{Case 2:} $b_{2m - 2} = 1$; $b_{2m - 1} = 0$. Then
\[
\psi_m^{\ell + 1}(1) = \frac{g_{\vec{B}}(2m - 2) + 1}{\|\vec{B}\|}.
\]
\textbf{Case 3:} $b_{2m - 2} = 0$; $b_{2m - 1} = 1$. Then
\[
\psi_m^{\ell + 1}(1) = \frac{g_{\vec{B}}(2m - 2) + 1}{\|\vec{B}\|^{\ell + 1}} + \sum_{n = 1}^\ell\frac{g_{\vec{B}}(2m - 2)}{\|\vec{B}\|^n} = \frac{g_{\vec{B}}(2m - 2) + 1}{\|\vec{B}\|^{\ell + 1}} + g_{\vec{B}}(2m - 2)\frac{\|\vec{B}\|^\ell - 1}{(\|\vec{B}\| - 1)\|\vec{B}\|^\ell}.
\]
\textbf{Case 4:} $b_{2m - 2} = b_{2m - 1} = 1$. Then
\[
\psi_m^{\ell + 1}(1) = \frac{g_{\vec{B}}(2m - 2) + 2}{\|\vec{B}\|^{\ell + 1}} + \sum_{n = 1}^\ell\frac{g_{\vec{B}}(2m - 2) + 1}{\|\vec{B}\|^n} = \frac{g_{\vec{B}}(2m - 2) + 2}{\|\vec{B}\|^{\ell + 1}} + (g_{\vec{B}}(2m - 2) + 1)\frac{\|\vec{B}\|^\ell - 1}{(\|\vec{B}\| - 1)\|\vec{B}\|^\ell}.
\]
In Algorithm 1, we have $g_{\vec{B}}(2m - 2) = 0$. If $\psi_m^{\ell + 1}(1) = 0$, then clearly $b_{2m - 2} = b_{2m - 1} = 0$; else
\[
\psi_m^{\ell + 1}(1) \in \left\{\frac{1}{\|\vec{B}\|},\frac{1}{\|\vec{B}\|^{\ell+1}},\frac{1}{\|\vec{B}\| - 1} + \frac{\|\vec{B}\| - 2}{(\|\vec{B}\| - 1)\|\vec{B}\|^{\ell + 1}}\,\middle|\,2\leq \|\vec{B}\| \leq N - 1\right\}.
\]
Using some basic algebra, we note that for $\ell \geq 1$,
\[
\left\{\frac{1}{\|\vec{B}\|}\,\middle|\,2\leq \|\vec{B}\|\leq N - 1\right\}\bigcap\left\{\frac{1}{\|\vec{B}\| - 1} + \frac{\|\vec{B}\| - 2}{(\|\vec{B}\| - 1)\|\vec{B}\|^{\ell + 1}}\,\middle|\,2\leq \|\vec{B}\|\leq N - 1\right\} = \emptyset
\]
since the numbers are properly interlaced
\[
\sum_{n=1}^\ell\frac{1}{2^n}+\frac{2}{2^{\ell+1}} = 1 > \sum_{n=1}^\ell\frac{1}{3^n}+\frac{2}{3^{\ell+1}} > \frac{1}{2} > \sum_{n=1}^\ell\frac{1}{4^n}+\frac{2}{4^{\ell+1}} > \frac{1}{3} > ... > \frac{1}{N-1}.
\]
Thus, it suffices to find an integer $L$ such that for $\ell \geq L$,
\[
\left\{\frac{1}{\|\vec{B}\|},\frac{1}{\|\vec{B}\| - 1} + \frac{\|\vec{B}\| - 2}{(\|\vec{B}\| - 1)\|\vec{B}\|^{\ell + 1}}\,\middle|\,2\leq \|\vec{B}\| \leq N - 1\right\}\bigcap\left\{\frac{1}{\|\vec{B}\|^{\ell + 1}}\,\middle|\,2\leq \|\vec{B}\|\leq N - 1\right\} = \emptyset.
\]
The simplest way to find such an $L$ is to take the smallest positive integer $L$ such that $2^{L+1} > N - 1$. \\\\
It follows that we can then determine the parameters $(b_{2m-2},b_{2m-1},\|\vec{B}\|) \in \{0,1\}\times\{0,1\}\times\{2,3,...,N-1\}$.
\\\\
In the validation of Algorithm 2, it is equivalent to consider the three situations:
\\\\
Situation 1. $g_{\vec{B}}(2m - 2) = 0$
\\\\
Situation 2. $g_{\vec{B}}(2m - 2) = \|\vec{B}\| - 1$
\\\\
Situation 3. $0 < g_{\vec{B}}(2m - 2) < \|\vec{B}\| - 1$
\\\\
As for situation 1, we just showed that we may solve for $b_{2m - 2}$ and $b_{2m - 1}$. It is clear that $b_{2m - 2} = b_{2m - 1} = 0$ in situation 2 since Algorithm 1 identified a nonzero digit. Under the assumption of situation 3, we have that all of the values of $\psi_m^{\ell + 1}(1)$ in cases 1 through 4 are distinct. This follows from tedious algebra, so we only show that Case 2 and Case 3 are different and leave the remainder to the reader to verify. Since $g_{\vec{B}}(2m - 2) < \|\vec{B}\| - 1$, we have $(g_{\vec{B}}(2m - 2) + 1)(\|\vec{B}\|^\ell - 1) + \|\vec{B}\| < \|\vec{B}\|^{\ell + 1}$. Rearranging and combining terms, we find
\[
(g_{\vec{B}}(2m - 2) + 1)(\|\vec{B}\| - 1) + g_{\vec{B}}(2m - 2)(\|\vec{B}\|^\ell - 1)\|\vec{B}\| < (g_{\vec{B}}(2m - 2) + 1)(\|\vec{B}\| - 1)\|\vec{B}\|^\ell.
\]
We conclude that Case 2 and Case 3 are distinct from dividing through by $(\|\vec{B}\| - 1)\|\vec{B}\|^{\ell + 1}$.

\end{proof}

\begin{remark}
The sampling set
\[
\left\{\frac{2m}{N^{\ell+1}}+\sum_{n=1}^\ell\frac{2m-1}{N^n}\,\middle|\,m \in \left\{1,2,...,\left\lfloor\frac{N}{2}\right\rfloor\right\}\right\}
\]
completely determines $\vec{B}$ up to ambiguity of the last nonzero digit in $\vec{B}$. That is, suppose that for some $m\in\big\{1,2,...,\big\lfloor \frac{N}{2}\big\rfloor\big\}$, we have that $b_n=0$ for all $n>2m$. Then there is ambiguity in the binary digit vector elements $(b_{2m-2},b_{2m-1})$ as they could be either $(1,0)$ or $(0,1)$ and the samples would agree.
\end{remark}

\subsubsection{Rationality and the CDF}
\begin{lemma}\label{RatToRat}
Let $\vec{B}$ be a digit set of length $N$. If $x \in \mathbb{Q} \cap [0,1]$, then $F_{\vec{B}}(x) \in \mathbb{Q}$.
\end{lemma}

\begin{proof} We first note that $F_{\vec{B}}(0)=0$ and $F_{\vec{B}}(1)=1$.

Then let $x \in \mathbb{Q}\cap(0,1)$, and consider its $N$-adic representation $x = \sum\limits_{i=1}^\infty\frac{n_i}{N^i}$
where $n_i \in \{0,1,...,N-1\}$. Since $x$ is rational, the sequence $\{n_i\}_{i=1}^\infty$ is eventually periodic. Recall from Proposition \ref{NtoDrel} that
\[
F_{\vec{B}}(x) = \sum_{i=1}^\infty\left(\prod_{k=1}^{i-1}b_{n_k}\right)\frac{g_{\vec{B}}(n_i)}{\|\vec{B}\|^i}.
\]
If there exists a positive integer $\ell$ such that $b_{n_\ell} = 0$, then
\[
F_{\vec{B}}(x) = \sum_{i=1}^\ell\left(\prod_{k=1}^{i-1}b_{n_k}\right)\frac{g_{\vec{B}}(n_i)}{\|\vec{B}\|^i},
\]
which is rational. Note that this is the case if $g_{\vec{B}}(n_i) = \|\vec{B}\|$ for some $i$ as we may then take $\ell = i + 1$. Otherwise, assume that $b_{n_k} = 1$ for all $k$. Then $g_{\vec{B}}(n_i) \in \{0,1,...,\|\vec{B}\|-1\}$ for all $i$, and we have the $\|\vec{B}\|$-adic representation,
\[
F_{\vec{B}}(x) = \sum_{i=1}^\infty\frac{g_{\vec{B}}(n_i)}{\|\vec{B}\|^i}.
\]
Since the sequence $\{g_{\vec{B}}(n_i)\}_{i=1}^\infty$ is eventually periodic, it follows that $F_{\vec{B}}(x)$ is rational.

\end{proof}

\begin{lemma}
Let $C_{\vec{B}}$ be a Cantor set and $F_{\vec{B}}$ the CDF.
For $x\in {\mathbb{Q}^c}\cap[0,1]$, $x\in C_{\vec{B}}$ if and only if $F_{\vec{B}}(x)\in\mathbb{Q}^c.$
\label{IrrToIrr}
\end{lemma}

\begin{proof}
Let $\vec{B}=(b_0,...,b_{N-1})$ be the binary representation of $F_{\vec{B}}$.\\ 
Suppose $x\in{\mathbb{Q}^c}\cap [0,1]$.
Since $x\in(0,1)$ it follows $x=\sum_{i=1}^\infty\frac{n_i}{N^i}$, for some $\{n_i\}_{i=1}^\infty$. Further, since $x$ is irrational, $\{n_i\}_{i=1}^\infty$ is never periodic. By Proposition \ref{NtoDrel}, $F_{\vec{B}}(x)=\sum_{i=1}^\infty\left(\prod_{k=1}^{i-1}b_{n_k}\right)\frac{g_{\vec{B}}({n_i})}{\|\vec{B}\|^i}$.\\ 
Suppose $x\in C_{\vec{B}}$. Since $x\in C_{\vec{B}}$, it follows $n_k\in D$, $b_{n_k}=1$, and $g(n_k)\in\{0,1,...,\|\vec{B}\|-1\}$ for all $k$. Then, $F_{\vec{B}}(x)=\sum_{i=1}^\infty\frac{g_{\vec{B}}({n_i})}{\|\vec{B}\|^i}$.\\
Note, $g(j+1)>g(j)$ whenever $j\in D$ implies $g_{\vec{B}}|_D$ is injective.
Then, since $\{n_i\}_{i=1}^\infty$ is never periodic and $n_i\in D$, it follows that $\{g_{\vec{B}}(n_i)\}_{i=1}^\infty$ is also never periodic. Then, $F_{\vec{B}}(x)$ is a never periodic decimal in base $\|\vec{B}\|$. Thus, $F_{\vec{B}}(x)\in\mathbb{Q}^c$.\\
Alternatively, suppose $x\not\in C_{\vec{B}}$. Then, there exists a smallest $K$ such that $n_K\not\in D$ and $b_{n_K}=0$. Then $\prod_{k=1}^{i-1}b_{n_k}=0$ if and only if $i>K$ and $F_{\vec{B}}(x)=\sum_{i=1}^K\frac{g_{\vec{B}}({n_i})}{\|\vec{B}\|^i}$. Thus, $F_{\vec{B}}(x)\in\mathbb{Q}$.
\end{proof}
\begin{cor}
If $x\not\in C_{\vec{B}}$, then $F_{\vec{B}}(x)\in\mathbb{Q}$.
\label{noCthenQ}
\end{cor}
\begin{proof}
Let $x\not\in C_{\vec{B}}$. If $x\in\mathbb{Q}$, by Lemma \ref{RatToRat}, $x\in\mathbb{Q}$. If $x\in\mathbb{Q}^c$, by Lemma \ref{IrrToIrr}, $x\in\mathbb{Q}$.
\end{proof}
\subsubsection{Multiplicatively Dependent Scale Factors}
\begin{lemma}
Let $F_{\vec{B}_1}$ be a CDF with scale factor $N^L$ and $F_{\vec{B}_2}$ be a CDF with scale factor $N^M$, for $L,M,N\in\mathbb{N}$. If $\vec{B}_1\otimes \vec{B}_2=\vec{B}_2\otimes \vec{B}_1$, then $F_{\vec{B}_1}=F_{\vec{B}_2}$.
\label{commutativekronecker}
\end{lemma}
\begin{proof}
We first note that the Kronecker product is associative. Let $\vec{B}_1\otimes \vec{B}_2=\vec{B}_2\otimes \vec{B}_1$.\\
By Corollary \ref{ISFcor3}, $F_{\vec{B}_1}=F_{\vec{B}_1^{\otimes L}}$ and $F_{\vec{B}_2}=F_{\vec{B}_2^{\otimes M}}$. Then, $\vec{B}_1^{\otimes L}$ and $\vec{B}_2^{\otimes M}$ have length $N^{LM}$.\\
We will show $\vec{B}_1^{\otimes L}\otimes\vec{B}_2^{\otimes M}=\vec{B}_2^{\otimes M}\otimes\vec{B}_1^{\otimes L}$, by first showing $\vec{B}_1^{\otimes L}\otimes \vec{B}_2= \vec{B}_2\otimes
\vec{B}_1^{\otimes L}$ by inducting on $L$.\\
As the base case, when $L=1$, $\vec{B}_1^{\otimes 1}\otimes \vec{B}_2=\vec{B}_1\otimes \vec{B}_2=\vec{B}_2\otimes \vec{B}_1=\vec{B}_2\otimes\vec{B}_1^{\otimes 1}$.\\
Now assume $\vec{B}_1^{\otimes L}\otimes \vec{B}_2=\vec{B}_2\otimes
\vec{B}_1^{\otimes L}$. Then, 
\begin{align*}
\vec{B}_1^{\otimes L+1}\otimes \vec{B}_2&=\vec{B}_1\otimes
\vec{B}_1^{\otimes L}\otimes \vec{B}_2=\vec{B}_1\otimes \vec{B}_2\otimes \vec{B}_1^{\otimes L}\\
&=\vec{B}_2\otimes
\vec{B}_1\otimes\vec{B}_1^{\otimes L}=\vec{B}_2\otimes\vec{B}_1^{\otimes L+1}.
\end{align*}
This proves $\vec{B}_1^{\otimes L}\otimes \vec{B}_2= \vec{B}_2\otimes
\vec{B}_1^{\otimes L}$.\\
Now we will induct on $M$. For the base case, when $M=1$,
$\vec{B}_1^{\otimes L}\otimes \vec{B}_2=\vec{B}_2\otimes\vec{B}_1^{\otimes L}$.\\
Now assume $\vec{B}_1^{\otimes L}\otimes \vec{B}_2^{\otimes M}=\vec{B}_2^{\otimes M}\otimes\vec{B}_1^{\otimes L}$. Then, 
\begin{align*}
    \vec{B}_2^{\otimes M+1}\otimes\vec{B}_1^{\otimes L}&=\vec{B}_2\otimes \vec{B}_2^{\otimes M}\otimes \vec{B}_1^{\otimes L}=\vec{B}_2\otimes \vec{B}_1^{\otimes L}\otimes \vec{B}_2^{\otimes M}\\
    &=\vec{B}_1^{\otimes L}\otimes \vec{B}_2\otimes \vec{B}_2^{\otimes M}=\vec{B}_1^{\otimes L}\otimes \vec{B}_2^{\otimes M+1}.
\end{align*}
By induction, $\vec{B}_1^{\otimes L}\otimes\vec{B}_2^{\otimes M}=\vec{B}_2^{\otimes M}\otimes\vec{B}_1^{\otimes L}$.
Since $\vec{B}_1^{\otimes L}$ and $\vec{B}_2^{\otimes M}$ have length $N^{LM}$,  $\vec{B}_1^{\otimes L}$ and $\vec{B}_2^{\otimes M}$ can be represented as $N^{LM}$ long row vectors. This gives an equivalent definition of the Kronecker product on matrices. 
Since
$\vec{B}_1^{\otimes L}\otimes\vec{B}_2^{\otimes M}=\vec{B}_2^{\otimes M}\otimes\vec{B}_1^{\otimes L}$, either
 $\vec{B}_1^{\otimes L}=c\vec{B}_2^{\otimes M}$ or $\vec{B}_2^{\otimes M}=c\vec{B}_1^{\otimes L}$, for some $c\in \mathbb{Z}_2$ (see Theorem 24 of \cite{Bro06a}).
If $c=0$, this implies $\vec{B}_1=0$ or $\vec{B}_2=0$, which is a contradiction.
Therefore, $c=1$, and  \[\vec{B}_1^{\otimes L}=\vec{B}_2^{\otimes M}.\]
Thus,  \[F_{\vec{B}_1}=F_{\vec{B}_1^{\otimes L}}=F_{\vec{B}_2^{\otimes M}}=F_{\vec{B}_2}.\] 

\end{proof}
\begin{lemma}
Let $\vec{A}$ have scale factor $N$, and $\vec{B}$, and $\vec{C}$ both have scale factor $M$. If $\vec{A}\otimes \vec{B}=\vec{A}\otimes \vec{C}$, then $\vec{B}=\vec{C}$.

\label{invCommute}
\end{lemma}
\begin{proof}
Let $\vec{A}=(a_0,...,a_{N-1})$, $\vec{B}=(b_0,...,b_{M-1})$, and $\vec{C}=(c_0,...,c_{M-1})$.
From $\vec{A}\otimes \vec{B}=\vec{A}\otimes \vec{C}$, it follows $a_ib_j=a_ic_j$ $\forall i,j$ such that $0\leq i\leq N-1$, $0\leq j\leq M-1$. Since $\vec{A}$ is a valid binary representation, $\vec{A}\neq 0$ so $\exists I$ such that $a_I\neq 0$. Then, $a_Ib_j=a_Ic_j$ $\forall j\in \{0,...,M-1\}$, and $b_j=c_j$ $\forall j\in \{0,...,M-1\}$. Thus, $\vec{B}=\vec{C}$.
\end{proof}

\begin{proposition}
Let $L,M,N\in\mathbb{N}$. Let $S=\{\frac{m}{N^{L+M}}\}_{m=1}^{N^{L+M}-1}$. Let $\vec{B}_L$ be a binary vector of a CDF with length $N^L$ and $\vec{B}_M$ be a binary vector of a CDF with length $N^M$. Then, $F_{\vec{B}_L}(x)=F_{\vec{B}_M}(x)$ for all $x\in S$ if and only if $F_{\vec{B}_L}=F_{\vec{B}_M}$.
\label{depFix}
\end{proposition}
\begin{proof}
Let $\vec{B}_L=(b_0,...,b_{N^L-1})$ and $\vec{B}_M=(c_0,...,c_{N^M-1})$. Let $g_L$ be the cumulative digit function for $\vec{B}_L$ 
and $g_M$ be the cumulative digit function for $\vec{B}_M$.\\
If $F_{\vec{B}_L}=F_{\vec{B}_M}$, clearly $F_{\vec{B}_L}(x)=F_{\vec{B}_M}(x)$ when $x\in S$.\\
Suppose that $F_{\vec{B}_L}(x)=F_{\vec{B}_M}(x)$ for all $x\in S$.\\
Let $F_{\vec{B}_L\otimes \vec{B}_M}$ be the CDF for $\vec{B}_L\otimes \vec{B}_M$,  and $g_{LM}$ be the cumulative digit function. Therefore, 
\[ {\|\vec{B}_L\otimes \vec{B}_M\|={\|\vec{B}_L\|} {\|\vec{B}_M\|}}. \]
Let $k\in\{0,...,N^L-1\}$.
By Lemma \ref{kroneckerg}, $g_{LM}(kN^M)=g_{L}(k){\|\vec{B}_M\|}$. 
\\
It follows from Proposition \ref{fixedpoints} that \[F_{\vec{B}_L\otimes \vec{B}_M}\left(\frac{k}{N^L}\right)=F_{\vec{B}_L\otimes \vec{B}_M}\left(\frac{kN^M}{N^{L+M}}\right)=\frac{g_{LM}(kN^M)}{{\|\vec{B}_L\|}{\|\vec{B}_M\|}}=\frac{g_L(k){\|\vec{B}_M\|}}{{\|\vec{B}_L\|}{\|\vec{B}_M\|}}=\frac{g_L(k)}{{\|\vec{B}_L\|}}=F_{\vec{B}_L}\left(\frac{k}{N^L}\right)=F_{\vec{B}_M}\left(\frac{k}{N^L}\right)\] since $\frac{k}{N^L}\in S$.\\
Then, if $F_{\vec{B}_L}\left(\frac{k}{N^L}\right)=F_{\vec{B}_L}\left(\frac{k+1}{N^L}\right)$, since all CDFs are increasing functions, $F_{\vec{B}_L\otimes \vec{B}_M}(x)=F_{\vec{B}_L}(x)=F_{\vec{B}_M}(x)$ for all $x\in \left[\frac{k}{N^L},\frac{k+1}{N^L}\right]$.\\
Next, suppose $F_{\vec{B}_L}\left(\frac{k}{N^L}\right)<F_{\vec{B}_L}\left(\frac{k+1}{N^L}\right)$. By Theorem \ref{basicCDF}, $b_k=1$, so by Lemma \ref{kroneckerg}, for $j<N^M$,
$g_{LM}(kN^M+j)=g_L(k){\|\vec{B}_M\|}+g_M(j)$. Also, $F_{\vec{B}_L}$ is self-similar on the interval $\left[\frac{k}{N^L},\frac{k+1}{N^L}\right]$. Let $x\in S\cap\left(\frac{k}{N^L},\frac{k+1}{N^L}\right)$. Then, $x=\frac{k}{N^L}+\frac{j}{N^{L+M}}$ for $j\in\{1,...,N^M-1\}$. It follows \[F_{\vec{B}_L}(x)=F_{\vec{B}_L}\left(\frac{k}{N^L}\right)+\frac{1}{{\|\vec{B}_L\|}}F_{\vec{B}_L}\left(N^L\frac{j}{N^{L+M}}\right)=F_{\vec{B}_L}\left(\frac{k}{N^L}\right)+\frac{1}{{\|\vec{B}_L\|}}F_{\vec{B}_L}\left(\frac{j}{N^{M}}\right).\] Since $\frac{j}{N^M}\in S$ and by Proposition \ref{fixedpoints},  \[F_{\vec{B}_L}(x)=F_{\vec{B}_L}\left(\frac{k}{N^L}\right)+\frac{1}{{\|\vec{B}_L\|}}F_{\vec{B}_M}\left(\frac{j}{N^M}\right)=\frac{g_L(k)}{{\|\vec{B}_L\|}}+\frac{1}{{\|\vec{B}_L\|}}\cdot \frac{g_M(j)}{{\|\vec{B}_M\|}}=\frac{g_L(k){\|\vec{B}_M\|}+g_M(j)}{{\|\vec{B}_L\|} {\|\vec{B}_M\|}}.\]\\
Next, by Proposition \ref{fixedpoints}, \[F_{\vec{B}_L\otimes \vec{B}_M}(x)=F_{\vec{B}_L\otimes \vec{B}_M}\left(\frac{k}{N^L}+\frac{j}{N^{L+M}}\right)=\frac{g_{LM}(kN^{M}+j)}{{\|\vec{B}_L\|}{\|\vec{B}_M\|}}=\frac{g_L(k){\|\vec{B}_M\|}+g_M(j)}{{\|\vec{B}_L\|}{\|\vec{B}_M\|}}=F_{\vec{B}_L}(x).\]\\

Therefore, $F_{\vec{B}_L\otimes \vec{B}_M}(x)=F_{\vec{B}_L}(x)$ for all $x\in S$.\\

Further, by switching $L$ and $M$ above, $F_{\vec{B}_M\otimes \vec{B}_L}(x)=F_{\vec{B}_M}(x)$ for all $x\in S$. However, $F_{\vec{B}_M}(x)=F_{\vec{B}_L}(x)$ for all $x\in S$. Therefore, $F_{\vec{B}_M\otimes \vec{B}_L}(x)=F_{\vec{B}_L\otimes \vec{B}_M}(x)$ for all $x\in S$, and both have scale factor $N^{L+M}$. By Corollary \ref{uniqueDet}, $\vec{B}_M\otimes \vec{B}_L=\vec{B}_L\otimes \vec{B}_M$. It follows by Lemma \ref{commutativekronecker}, $F_{\vec{B}_M}=F_{\vec{B}_L}$.
\end{proof}
\subsubsection{Almost nowhere intersection of Cantor Sets}
\begin{lemma}
Let $h,M,N \in \mathbb{N}$ such that $M\not\sim N$. Then, for all $L\in\mathbb{N}$, there exists a constant $\alpha(M,N)\in(0,1)$ dependent on $M$ and $N$ such that
\[\sum_{n=0}^{L-1}\prod_{k=1}^{\infty}|\cos(N^{-k}hM^n\pi)|\leq 2L^{1-\alpha(M,N)}.\]
\label{borrowedLemma}
\end{lemma}
\begin{proof}
From Lemma 5 of \cite{Sch62a}, translated in Lemma 1 of \cite{Pol88a}, there exists a constant $\beta(M,N)>0$ dependent upon $M$ and $N$ such that \[\sum_{n=0}^{L-1}\prod_{k=1}^{\infty}|\cos(N^{-k}hM^n\pi)|\leq 2L^{1-\beta(M,N)}.\] Since $f(x)=2L^{1-x}$ is a decreasing function, if it is true for $\beta(M,N)\geq 1$, then it must also be true for some $\alpha(M,N)<1$. Then, letting
\[ \alpha(M,N)=\begin{cases} 
      \beta(M,N) & \beta(M,N)<1 \\
      \frac{1}{2} & \beta(M,N)\geq 1 \\
   \end{cases},
\]
it follows
\[\sum_{n=0}^{L-1}\prod_{k=1}^{\infty}|\cos(N^{-k}hM^n\pi)|\leq 2L^{1-\alpha(M,N)}.\]
\end{proof}

\begin{lemma}
Let $N, t \in\mathbb{N}$. Let $D=\{\epsilon_0,...,\epsilon_{d-1}\}\subset\mathbb{Z}_N$ (and $d\geq 2$). Then for any $j\in\mathbb{N}$
\[\left|\frac{1}{d}\sum_{k=0}^{d-1}e^{2\pi i\frac{t}{N^j}\epsilon_k}\right| \leq \left|\cos\left({\pi t\frac{|\epsilon_{a}-\epsilon_{b}|}{N^j}}\right)\right|\]
where $\epsilon_a,\epsilon_b \in D$ are such that \[\left |e^{2\pi i\frac{t}{N^j}\epsilon_a}+e^{2\pi i\frac{t}{N^j}\epsilon_b}\right|=\max_{l,m\in\{0,...,d-1\}, l\not=m}\left |e^{2\pi i\frac{t}{N^j}\epsilon_l}+e^{2\pi i\frac{t}{N^j}\epsilon_m}\right|.\]
\label{keithLemma}
\end{lemma}

\begin{proof}
    \begin{align*}
     \left|\frac{1}{d}\sum_{k=0}^{d-1}e^{2\pi it\frac{t}{N^j}\epsilon_k}\right|&=\left|\frac{1}{d}\cdot\frac{1}{2(d-1)}\sum_{k=0}^{d-1}\sum_{n\neq k}\bigg(e^{2\pi i\frac{t}{N^j}\epsilon_k}+e^{2\pi i \frac{t}{N^j}\epsilon_n}\bigg)\right| \\    &\leq\frac{1}{2d(d-1)}\left( \, \left|e^{2\pi i\frac{t}{N^j}\epsilon_0}+e^{2\pi i\frac{t}{N^j}\epsilon_1}\right|+...+\left|e^{2\pi i\frac{t}{N^j}\epsilon_{d-2}}+e^{2\pi i\frac{t}{N^j}\epsilon_{d-1}}\right|\,\right) \\
     &\leq\frac{1}{2d(d-1)}(d(d-1)) \left|e^{2\pi i\frac{t}{N^j}\epsilon_{a}}+e^{2\pi i\frac{t}{N^j}\epsilon_{b}}\right| \\
     &=\frac{1}{2}\cdot\left|e^{2 \pi i \frac{t}{N^j}\left(\frac{\epsilon_{a}+\epsilon_{b}}{2}\right)}\right|\left|e^{2 \pi i \frac{t}{N^j}\left(\frac{\epsilon_{a}-\epsilon_{b}}{2}\right)}+e^{2 \pi i \frac{t}{N^j}\left(\frac{-(\epsilon_{a}-\epsilon_{b})}{2}\right)}\right|\\
     &=1\cdot\left|\cos\left({2\pi \frac{t}{N^j}\frac{\epsilon_{a}-\epsilon_{b}}{2}}\right)\right|= \left|\cos\left(\pi t \frac{|\epsilon_a-\epsilon_b|}{N^j}\right)\right|
    \end{align*}  
\end{proof}

\begin{lemma}
Let $C_{\vec{B}}$ be a Cantor set with scale factor $N$ and digit set $D=\{\epsilon_0,...,\epsilon_{d-1}\}$. Let $M,L,h \in\mathbb{N}$ with $M>1$ and $M\not\sim N$. Let $\alpha(d,M)$ be defined as in Lemma \ref{borrowedLemma} and let $\delta=\frac{\alpha(d,M)}{3}$. Then the set of $x\in C_{\vec{B}}$ such that \[\left|\sum_{n=0}^{L-1}e(h M^n x)\right|\geq L^{1-\delta}\] has $\mu_{\vec{B}}$-measure of at most $6L^{-\delta}$.
\label{upperBoundMu}
\end{lemma}
\begin{proof}
Adapted from Lemma 3 of \cite{Cas59a}.\\
Note, $e(\lambda x)$ is defined and continuous on the interval $0\leq x\leq 1$. Let $Z_t$ be the set of non-negative integers less than $N^t$ containing only digits in $D$ in their base $N$ expansion. Therefore, by the invariance equation as applied to the push-forward measure $\mu_{\vec{B}}(A)=m(F_{\vec{B}}(A))$ where $m$ is Lebesgue measure, we can calculate
\[\int_{x\in C_{\vec{B}}}e(\lambda x)d\mu_{\vec{B}} = \lim_{t\rightarrow\infty}d^{-t}\sum_{z\in Z_t}e(\lambda N^{-t}z) = \lim_{t\rightarrow
\infty}\prod_{j<t}\left(\frac{1}{d}\sum_{i=0}^{d-1} e(N^{-j}\lambda \epsilon_i) \right).\]
The details of the above calculation are given in \cite{Cas59a}.

Let $\epsilon_a, \epsilon_b\in D$ be such that \[\left |e^{2\pi i\frac{t}{N^j}\epsilon_a}+e^{2\pi i\frac{t}{N^j}\epsilon_b}\right|=\max_{l,m\in\{0,...,d-1\}, l\not=m}\left |e^{2\pi i\frac{t}{N^j}\epsilon_l}+e^{2\pi i\frac{t}{N^j}\epsilon_m}\right|.\] Let $r=|\epsilon_{a}-\epsilon_{b}|$. Note, $r\in\mathbb{N}$, since $|D|\geq 2$ and contains only integers.
Therefore, following from Lemma \ref{keithLemma}, \[\left|\int_{x\in C_{\vec{B}}}e(\lambda x)d\mu_{\vec{B}}\right|\leq\prod_{j=0}^\infty |\cos(N^{-j}\lambda r\pi)|.\]\\
Further, $|z|^2=z\overline{z}$ and $\overline{e(hM^nx)}=e(h(-M^n)x)$, so for any $L\in\mathbb{N}$
\begin{align*}
\int_{x\in C_{\vec{B}}}\left|\sum_{n=0}^{L-1}e(h M^n x)\right|^2d\mu_{\vec{B}} &=\left| \sum_{m=0}^{L-1}\sum_{n=0}^{L-1}\int_{x\in C_{\vec{B}}} e(h(M^n-M^m)x)d\mu_{\vec{B}}\right|\\
&\leq\sum_{m=0}^{L-1}\sum_{n=0}^{L-1}\left| \int_{x\in C_{\vec{B}}} e(h(M^n-M^m)x)d\mu_{\vec{B}}\right|\\
&\leq\sum_{m=0}^{L-1}\sum_{n=0}^{L-1}\prod_{j=0}^\infty \left|\cos(N^{-j}h(M^n-M^m)r\pi)\right|.
\end{align*}
Consider $l=\min(m,n)$ and $k=\max(m,n)-l$, so that $l$ and $k$ determine $m$ and $n$ up to pairs. It follows, \[\sum_{m=0}^{L-1}\sum_{n=0}^{L-1}\prod_{j=0}^\infty |\cos(N^{-j}h(M^n-M^m)r\pi)|\leq 2\sum_{l=0}^{L-1}\sum_{k=0}^{L-1}\prod_{j=0}^\infty |\cos(N^{-j}h(M^l-1)M^k r\pi)|.\] 
When $l=0$, all terms in the product are $\cos(0)=1$ and the inner sum is no more than $L$. Otherwise, by Lemma \ref{borrowedLemma}, the inner sum is less than or equal to $2L^{1-\alpha(d,M)}$ where $\alpha(d,M)>0$. \\
 Therefore, \[\int_{x\in C_{\vec{B}}}\left|\sum_{n=0}^{L-1}e(h M^n x)\right|^2\leq 2(L+L(2L^{1-\alpha(d,M)}))=2L+4L^{2-\alpha(d,M)}<6L^{2-\alpha(d,M)}=6L^{2-3\delta}.\]\\
It follows that the $\mu_{\vec{B}}$-measure of $x\in C_{\vec{B}}$ such that $\left|\sum_{n=0}^{L-1}e(h M^n x)\right|\geq L^{1-\delta}$ is no more than $\frac{6L^{2-3\delta}}{L^{2(1-\delta)}}=6L^{-\delta}$ by Chebychev's inequality.
\end{proof}

Our proof of the following theorem is adapted from \cite{Cas59a}.

\begin{theorem}
Let $C_{\vec{B}}$ be a Cantor set.  Then $\mu_{\vec{B}}$-almost all $x\in C_{\vec{B}}$ are normal to every base $M>1$ such that $M\not\sim N$.
\label{cantorNormal}
\end{theorem}
\begin{proof}
Fix $M\in\mathbb{N}$. Let $\delta=\frac{\alpha(d,M)}{3}$, where $\alpha$ is defined as in Lemma \ref{borrowedLemma}. Then, $0<\delta<1$.
Let $L_j=\lfloor e^{2\sqrt{j}}\rfloor$. Then, for $j>1$, $L_j^{-\delta}\leq e^{-\delta\sqrt{j}}$, and $\int_{0}^\infty e^{-\delta\sqrt{j}} dj<\infty$, so 
$\sum_{j= 0}^\infty L_j^{-\delta}<\infty$. It follows for every $\epsilon>0$ there exists $J_\epsilon\in\mathbb{N}$ such that \[\sum_{j=J_\epsilon}^\infty 6L_j^{-\delta}<\epsilon.\]
By Lemma \ref{upperBoundMu}, the sum of the $\mu_{\vec{B}}$-measures of the sets $\left\{x : \left|\sum_{n=0}^{L_j-1}e( M^n x)\right|\geq L_j^{1-\delta}\right\}$ for some $j\geq J$ goes to 0 as $J\rightarrow\infty$. Therefore, for $\mu_{\vec{B}}$-almost all $x$ there exists $J_x$ such that \[\left|\sum_{n=0}^{L_j-1}e( M^n x)\right|<L_j^{1-\delta}\text{ for all $j\geq J_x$,}\] so $\left|\sum_{n=0}^{L_j-1}e( M^n x)\right|=o(L_j)$ as $L_j\rightarrow\infty$.

Further, for every $L$ there exists $j_L$ such that $L_{j_L}\leq L < L_{j_L+1}$.
 In addition, \[\left|\sum_{n=0}^{L-1}e( M^n x)-\sum_{n=0}^{L_{j_L}-1}e(M^n x)\right|\leq L-L_{j_L}.\]
 Note, $L-L_{j_L}=o(L)$ as $L\rightarrow\infty$, because $L_{j}$ grows slower than a geometric series.
 Therefore, \[\sum_{n=0}^{L-1}e( M^n x)=o(L)\] as $L\rightarrow\infty$ for $\mu_{\vec{B}}$-almost all $x$.
 
 Note, the set of $x$ such that $\sum_{n=0}^{L-1}e( M^n x)\not=o(L)$ for a fixed $M\in\mathbb{N}$ has $\mu_{\vec{B}}$-measure 0, and the sets of possible $h$ and $M$ are countable. Then, the set of $x$ such that $\sum_{n=0}^{L-1}e( M^n x)\not=o(L)$ for any $M\in\mathbb{N}$ has $\mu_{\vec{B}}$-measure 0 because it is the union of a countable number of sets with $\mu_{\vec{B}}$-measure 0. For $\mu_{\vec{B}}$-almost all $x$, $\sum_{n=0}^{L-1}e(h M^n x)=o(L)$ for all $M\in\mathbb{N}$ such that $M\not\sim N$. By Weyl's criterion in \cite{Wey16a}, then for $\mu_{\vec{B}}$-almost all $x$ and any fixed $M\not\sim N$, the fractional part of the sequence $\{M^n x\}_{n=1}^\infty$ is uniformly distributed. Therefore, $\mu_{\vec{B}}$-almost all $x$ are normal to all bases $M>1$ such that $M\not\sim N$.
\end{proof}

\begin{theorem}
Let $M,N$ be scale factors of the Cantor sets $C_{\vec{B}},C_{\vec{C}}$, respectively. If $M\not\sim N$, then $C_{\vec{B}}\cap C_{\vec{C}}$ is $\mu_{\vec{B}}$-almost empty and $\mu_{\vec{C}}$-almost empty.
\label{emptyInter}
\end{theorem}
\begin{proof}
Let $M,N$ be scale factors of the Cantor sets $C_{\vec{B}},C_{\vec{C}}$, respectively, and $M\not\sim N$.\\
By Theorem \ref{cantorNormal}, $\mu_{\vec{B}}$-almost all of the elements in $C_{\vec{B}}$ are normal in base $M$. Since normal numbers contain all of the digits, it follows $\mu_{\vec{B}}$-almost all of the elements in $C_{\vec{B}}$ are not in $C_{\vec{C}}$. Similarly, $\mu_{\vec{C}}$-almost all of the elements in $C_{\vec{C}}$ are normal in base $N$ and likewise are not elements of $C_{\vec{B}}$. Therefore, it follows their intersection is $\mu_{\vec{C}}$-almost empty and $\mu_{\vec{B}}$-almost empty.
\end{proof}

Note, normality is a stronger condition than necessary to show an element is not in any Cantor set with scale factor $N$. In fact, it only must have every digit appear at least once.
\begin{cor}
For every Cantor set $C_{\vec{B}}$, there exist irrational numbers in $C_{\vec{B}}$ normal to every base $M$ such that $M\not\sim N$.
\label{irrInC}
\end{cor}
\begin{proof}
There are uncountably many elements in $C_{\vec{B}}$, however, there are only countably many rationals. Further, by Theorem \ref{cantorNormal}, $\mu_{\vec{B}}$-almost all of the elements in $C_{\vec{B}}$ are normal to multiplicatively independent bases. It follows that $\mu_{\vec{B}}$-almost all of the elements in $C_{\vec{B}}$ must be irrational and normal in multiplicatively independent bases.
\end{proof}
\subsubsection{Using Samples}
\begin{lemma}
Let $C_{\vec{B}}$ be a Cantor set with scale factor $N$. Let $D_1,...,D_k$ be all possible digit sets of $N$ with cardinality 2, and associate $\vec{B}_1,...,\vec{B}_k$ with such digit sets. Then for any set of irrationals $\{x_i\in C_{\vec{B}_i}:i=1,...,k\}$, $D$ can be uniquely determined from $\{(x_i,F_{\vec{B}}(x_i))\}_{i=1}^{k}$. In particular, $D=\bigcup_{i\in A} D_i$ where $A=\{i \mid F_{\vec{B}}(x_i)\in\mathbb{Q}^c\}$  
\label{detDigitSet}
\end{lemma}
\begin{proof}
Let $D_1,...,D_k$ be all the digits sets of cardinality 2 for scale factor $N$. Let $\vec{B}=(b_0,...,b_{N-1})$ be the binary representation of $C_{\vec{B}}$.\\
 Consider $x_i\in C_{\vec{B}_i}$. Since $x_i$ is irrational and $|D_i|=2$, the decimal expansion of $x_i$ in base $N$ must contain both digits in $D_i$. Then, $x_i\in C_{\vec{B}}$ if and only if $D_i\subseteq D$. Then, by Lemma \ref{IrrToIrr}, $F_{\vec{B}}(x_i)$ is irrational if and only if $D_i\subseteq D$.
Let $A=\{i \mid F_{\vec{B}}(x_i)\in\mathbb{Q}^c\}$. Then, $\bigcup_{i\in A} D_i \subseteq D$.
Next, consider $\epsilon_1\in D$. Since $\|\vec{B}\|\geq 2$, there exists an $\epsilon_2\in D$, $\epsilon_1\not=\epsilon_2$. Further, there exists a $j$ such that $D_j=\{\epsilon_1,\epsilon_2\}$. Then, $D_j\subseteq D$, $x_j\in C_{\vec{B}}$ and $F_{\vec{B}}(x_j)$ will be irrational by Lemma \ref{IrrToIrr}. It follows $j\in A$, and therefore $D_j\subseteq \bigcup_{a\in A} D_a $. Thus, $\epsilon_1\in \bigcup_{a\in A} D_a$. Since $\epsilon_1$ is arbitrary, it follows $\bigcup_{a\in A} D_a= D$ where $A=\{i \mid F_{\vec{B}}(x_i)\in\mathbb{Q}^c\}$.
\end{proof}
\begin{theorem}
Given $K$, there exists a constant $M:=M(K)$ such that there exists $\{x_i\}_{i=1}^{M(K)}\subset(0,1)$ which is a set of uniqueness for $\mathscr{G}_K$. The constant $M(K)=O(K^3)$.
\label{maintheorem}
\end{theorem}
\begin{proof}
Let $F_{\vec{B}}$ be a CDF with scale factor $N\leq K$. We proceed with $C_{N,D}$ (instead of $C_{\vec{B}}$ notation as it makes the proof clearer.

For every $M$, $3\leq M\leq K$, there exist $\binom{M}{2}=\frac{M(M-1)}{2}$ unique Cantor sets $C_{M,D_i}$ such that $|D_i|=2$. Further, by Corollary \ref{irrInC}, each of these Cantor sets contain an irrational element (in fact, almost all elements) which is normal to all bases multiplicatively independent of $M$.
Then, for every $M$ and $|D_i|=2$, there exists $x\in C_{M,D_i} \cap \mathbb{Q}^c$ such that $\forall L \not\sim M, L\leq K$, $x$ in base $L$ has all possible digits in its representation.

Choose one such element for each $C_{M,D_i}$, and denote it $x_{M,D_i}$; let \[S_1=\{x_{M,D_i} \mid 3\leq M\leq K, |D_i|=2\}\] Note, $|S_1|\leq\frac{K(K-1)(K-3)}{2}$.

For any $M\not\sim N$ and $D_i$ with $|D_i|=2$, $x_{M,D_i}$ contains every possible digit in $N$. Then, since any $x\in C_{\vec{B}}$ cannot contain every digit in base $N$, $x_{M,D_i}\not\in C_{\vec{B}}$. It follows from Corollary \ref{noCthenQ} that $F_{\vec{B}}(x_{M,D_i})\in\mathbb{Q}$.

Suppose now that there exists a CDF with scale factor $M$ passing through all of the points $\{(x_i, F_{\vec{B}}(x_i))\mid x_i\in S_1\}$. Since this is true for all $D_i$ corresponding with $M$, by Lemma \ref{detDigitSet} the digit set of the CDF is be empty, a contradiction. Thus, no such CDF exists and $M$ can be eliminated as a scale factor.

Thus, all possible scale factors remaining are multiplicatively dependent to $N$. Therefore, there is a fixed $J\in\mathbb{N}$ such that for each possible scale factor $N'$, $N'=J^{L_{N'}}$ for some $L_{N'}\in\mathbb{N}$. Further, by Lemma \ref{detDigitSet}, for each $N'$ there exists at most one digit set, $\vec{B'}$, such that $F_{\vec{B}}(x)=F_{\vec{B'}}(x)$ for all $x\in S_1$.

By Proposition \ref{depFix}, for all $L_1,L_2\in\mathbb{N}$ and $\vec{B}_1,\vec{B}_2$ with scale factors $J^{L_1},J^{L_2}$, respectively, either $F_{\vec{B}_1}=F_{\vec{B}_2}$ or only one agrees with $\left\{\left(\frac{m}{J^{L+M}}, F_{\vec{B}}\left(\frac{m}{J^{L+M}}\right)\right)\right\}_{m=1}^{J^{L+M}-1}$. Note that for any $J$, $2\leq J\leq K$, there is $L_K\in\mathbb{N}$ such that $J^{L_K}\leq K < J^{L_K+1}$. The set of rational numbers expressible with denominator $J^{2L_K}$ includes the set of rational numbers expressible with denominator $J^{L}$ for $L\leq 2L_K$. Note, $J^{L_1}, J^{L_2} \leq K$ implies $L_1+L_2\leq 2L_K$.

Since $J^{L_K}\in\{2,3,..,K\}$, sampling at \[S_2=\left\{\frac{m}{M}\right\}_{m=1, M\in\{2^2,3^2,4^2,...,K^2\}}^{M-1}\] is sufficient to differentiate all mutliplicatively dependent bases no more than $K$. Hence sampling at $\left\{\frac{m}{(J^{L_K})^2}\right\}_{m=1}^{\left(J^{L_K}\right)^2-1}$
is sufficient for differentiating all bases multiplicatively dependent to $J$. It follows, of the remaining CDFs, only CDFs equivalent to $F_{\vec{B}}$ will pass through all the points $\{(x,F_{\vec{B}}(x))\mid
x\in S_2\}$, and all non-equivalent CDFs can be eliminated.

For any remaining CDFs $F$, $F=F_{\vec{B}}$. Thus, $S=S_1\cup S_2$ is sufficient to reconstruct $F_{\vec{B}}$.

Finally, we note that since $|S_2|\leq 1^2+2^2+3^2+...+K^2=\frac{K(K+1)(2K+1)}{6}$, $|S|\leq\frac{K(K-1)(K-3)}{2}+\frac{K(K+1)(2K+1)}{6}=\frac{5}{6}K^3-\frac{3}{2}K^2+\frac{5}{3}K=O(K^3)$.

\end{proof}

\begin{cor}
There exists a set of uniqueness for $\mathscr{G}_K$ with sample complexity $O(K^3)$.
\end{cor}

\begin{remark}
CDFs equivalent to $F_{\vec{B}}$  will not be eliminated by the algorithm described in Theorem \ref{maintheorem}, which only eliminates CDFs which do not pass through all the points. Then, the algorithm will produce all equivalent CDFs with scale factor less than $K$, which includes the CDF with the smallest possible scale factor, and the smallest possible scale factor can be determined.
\end{remark}
\begin{remark}
Since CDFs are equivalent only if their underlying Cantor sets are equal, the algorithm also reconstructs the underlying Cantor set $C_{\vec{B}}$.
\end{remark}
\section{Conclusion and Future Research}
With a upper scale factor bound of $K$, and $O(K^3)$ points, a CDF of any Cantor set can be completely reconstructed. While a minimum number of points has not been determined, there is a lower bound dependent upon the maximum possible scale factor. Further, many of the points sampled in Theorem \ref{maintheorem}, those in $S_1$, are not specific. Almost all of the points in the given Cantor set will suffice. 

If the scale factor $N$ is known, then $N-1$ well chosen points is enough to determine the digit set $D$. However, this is not the minimum number. A future research question would be to determine the minimum number of points necessary to determine the digit set. 

\section{Acknowledgements}

Allison Byars, Evan Camrud, Sarah McCarty, and Keith Sullivan were supported in part by the National Science Foundation under award \#1457443.

Steven Harding was supported in part by the National Geospatial-Intelligence Agency under award \#1830254.

Eric Weber was supported in part by the National Science Foundation under award \#1934884 and the National Geospatial-Intelligence Agency under award \#1830254.

\bibliographystyle{amsplain}
\bibliography{cdfbib}
\nocite{*}

\section{Appendix}

\begin{algorithm}\label{alg1}
\caption{Finding $\|\protect\vec{B}\|$ and the first nonzero digit of $\protect\vec{B}$}
Since $g_{\vec{B}}(k)=\sum_{j=0}^{k-1} b_j$ we have that $g(2m-2)=\sum_{j=0}^{2m-3} b_j$ such that we know $g(2m-2)$ recursively dependent on sample values and their determination of each $b_j$. We further recall that $\ell$ is a large enough integer, and that the smallest positive $\ell$ such that $2^{\ell + 1} > N - 1$ is sufficient.

\begin{algorithmic}
\STATE Initialize $m = 1$
\WHILE{$g_{\vec{B}}(2m - 2) = 0$}
\STATE Sample $F_{\vec{B}}$ at $$x = \frac{2m}{N^{\ell + 1}} + \sum_{n = 1}^\ell\frac{2m - 1}{N^n}.$$
\IF{$F_{\vec{B}}(x) = 0$}
\STATE $b_{2m - 2} = 0$
\STATE $b_{2m - 1} = 0$
\STATE $m = m + 1$
\ELSE
    \IF{$F_{\vec{B}}(x) = 1/\mathcal{\vec{B}}$ for some $2 \leq \mathcal{\vec{B}} \leq N - 1$}
    \STATE $\|\vec{B}\| = \mathcal{\vec{B}}$
    \STATE $b_{2m - 2} = 1$
    \STATE $b_{2m - 1} = 0$
    \ELSIF{$F_{\vec{B}}(x) = 1/\mathcal{\vec{B}}^{\ell+1}$ for some $2 \leq \mathcal{\vec{B}} \leq N - 1$}
    \STATE $\|\vec{B}\| = \mathcal{\vec{B}}$
    \STATE $b_{2m - 2} = 0$
    \STATE $b_{2m - 1} = 1$
    \ELSE
    \STATE There is an integer $2 \leq \mathcal{\vec{B}} \leq N - 1$ such that $$F_{\vec{B}}(x) = \frac{2}{\mathcal{\vec{B}}^{\ell + 1}} + \sum_{n=1}^\ell\frac{1}{\mathcal{\vec{B}}^n}$$
    \STATE $\|\vec{B}\| = \mathcal{\vec{B}}$
    \STATE $b_{2m - 2} = 1$
    \STATE $b_{2m - 1} = 1$
    \ENDIF
    \STATE Break
\ENDIF
\ENDWHILE
\IF{$N = 2m$}
\STATE Return
\ELSIF{$N = 2m + 1$}
    \IF{$g_{\vec{B}}(N - 1) = \|\vec{B}\|$}
    \STATE $b_{N - 1} = 0$
    \ELSE
    \STATE $b_{N - 1} = 1$
    \ENDIF
\ELSE
\STATE Proceed to Algorithm 2
\ENDIF
\end{algorithmic}
\end{algorithm}

\begin{algorithm}
\caption{Finding the remaining digits of $\protect\vec{B}$}
\begin{algorithmic}
\STATE Initialize $M = \left\lfloor\frac{N}{2}\right\rfloor$
\FOR{$\mathfrak{m} = 1,2,...,M - m$}
\IF{$g_{\cev{B}}(2m - 2) = \|\vec{B}\| - 1$}
\STATE $b_{N - 1 - (2\mathfrak{m} - 2)} = 0$
\STATE $b_{N - 1 - (2\mathfrak{m} - 1)} = 0$
\ELSE
\STATE Sample $F_{\cev{B}}$ at $$x = \frac{2\mathfrak{m}}{N^{\ell+1}} + \sum_{n=1}^\ell\frac{2\mathfrak{m} - 1}{N^n}$$
    \IF{$F_{\cev{B}}(x) = \frac{g_{\cev{B}}(2\mathfrak{m} - 2)}{\|\vec{B}\|}$}
    \STATE $b_{N - 1 - (2\mathfrak{m} - 2)} = 0$
    \STATE $b_{N - 1 - (2\mathfrak{m} - 1)} = 0$
    \ELSIF{$F_{\cev{B}}(x) = \frac{g_{\cev{B}}(2\mathfrak{m} - 2) + 1}{\|\vec{B}\|}$}
    \STATE $b_{N - 1 - (2\mathfrak{m} - 2)} = 1$
    \STATE $b_{N - 1 - (2\mathfrak{m} - 1)} = 0$
    \ELSIF{$$F_{\cev{B}}(x) = \frac{g_{\cev{B}}(2\mathfrak{m} - 2)}{\|\vec{B}\|^{\ell + 1}} + \sum_{n = 1}^\ell\frac{g_{\cev{B}}(2\mathfrak{m} - 3)}{\|\vec{B}\|^n}$$}
    \STATE $b_{N - 1 - (2\mathfrak{m} - 2)} = 0$
    \STATE $b_{N - 1 - (2\mathfrak{m} - 1)} = 1$
    \ELSE
    \STATE $b_{N - 1 - (2\mathfrak{m} - 2)} = 1$
    \STATE $b_{N - 1 - (2\mathfrak{m} - 1)} = 1$
    \ENDIF
\ENDIF
\ENDFOR
\IF{$N$ is odd}
    \IF{$$\sum_{n = 0}^{2m - 1}b_n + \sum_{n = 2m + 1}^{N - 1}b_n = \|\vec{B}\|$$}
    \STATE $b_{2m} = 0$
    \ELSE
    \STATE $b_{2m} = 1$
    \ENDIF
\ENDIF
\end{algorithmic}
\end{algorithm}

\bibliography{cdfbib.bib}

\end{document}